\documentclass[preprint,12pt]{elsarticle}

\usepackage{amsthm}
\usepackage{amsmath}
\usepackage{amsfonts,amssymb,amscd}
\usepackage{mathrsfs,mathtools}
\usepackage{latexsym}
\usepackage{bm}

\usepackage[figurename=Fig.]{caption}
\usepackage{subcaption}
\usepackage{graphicx}
\usepackage[table,xcdraw]{xcolor}
\usepackage{booktabs}
\usepackage{multirow}
\usepackage{rotating,tabularx}
\usepackage{float}
\usepackage{tikz}
\usepackage{tkz-euclide}
\usepackage{tkz-berge}

\usepackage{xparse}
\usepackage{needspace}
\usepackage{enumerate,longtable,tabu}
\usepackage[margin=2.5cm]{geometry}
\usepackage[hidelinks]{hyperref}

\graphicspath{{./img/}  {./figures}  }
\usetikzlibrary{topaths,calc,decorations.pathreplacing,hobby,fit,scopes,matrix,positioning,decorations.pathmorphing,automata,backgrounds}  

\tikzset{
 dot/.style={draw, circle, inner sep=0.5pt, fill=blue}  
}  

\pgfkeys{
    /rootG/.is family, /rootG,
    default/.style = {rotate=0, scale=0.8}, 
    rotate/.estore in = \rootGRotate,
    scale/.estore in = \rootGScale,
}  

\DeclareDocumentCommand\rootG{O{}  m}  {
    \pgfkeys{/rootG, default, #1}  
    \begin{scope}  [shift={#2}, rotate=\rootGRotate,scale=\rootGScale,thick]
        \draw[fill=cyan!10!white,fill opacity=0.5] (0,0) .. controls (-1.2,1.5) and (-1.2,2.5) .. (0,2.5) 
            .. controls (1.2,2.5) and (1.2,1.5) .. (0,0);
        \node at (0,1.5) {$G_r$}  ;
    \end{scope}  
}  
\DeclareDocumentCommand\rootGs{O{}  m}  {
    \pgfkeys{/rootG, default, #1}  
    \begin{scope}  [shift={#2}, rotate=\rootGRotate,scale=\rootGScale,thick]
        \draw[fill=cyan!10!white,fill opacity=0.5] (0,0) .. controls (-1.2,1.5) and (-1.2,2.5) .. (0,2.5) 
            .. controls (1.2,2.5) and (1.2,1.5) .. (0,0);
        \node at (0,1.5) {$G_{r_4}  $}  ;
    \end{scope}  
}

\DeclareDocumentCommand\longpath{m m m g g}  {
\def \n {#3}    \def \startn {#1}    \def \endn {#2}  
\IfNoValueTF{#5}  {\def \p {p}  }  {\def \p {#5}   }  
\foreach \v in {0,...,\n}   {\coordinate (\p\v) at ($(\startn)!\v/\n!(\endn)$);}  
\IfNoValueTF{#4}  
{\foreach \x [evaluate=\x as \y using int(\x-1)] in {1,...,\n}  {\draw[black](\p\y) node {}  --(\p\x) node {}  ;}  }  
{\foreach \x [evaluate=\x as \y using int(\x-1)] in {1,...,\n}  {
\IfValueInList \x {#4}  
{\foreach \v in {4,5,6}  {\fill ($ (\p\y)!\v/10! (\p\x) $) circle (0.5pt);}  }  
{\draw[black](\p\y) node {}  --(\p\x) node {}  ;}  
}  }  }

\DeclareDocumentCommand\nstar{o m m m m g g}  {
\def \o {#2}    \def \n {#3}    \def \deg {#4}    \def \len {#5}  
\IfNoValueTF{#1}  {\def \rot {0}  }  {\def \rot {#1}   }  
\IfNoValueTF{#7}  {\def \p {p}  }  {\def \p {#7}   }  
\foreach \x in {0,1,...,\n}   {\coordinate[rotate around={-0.5*\deg+\rot: (\o)}  ](\p\x) at ($(\o)+({\deg * \x/\n}  : \len)$);}  
\foreach \x[evaluate=\x as \y using int(\x+1),evaluate=\x as \z using int(\x-1)] in {0,1,...,\n}  {
  \IfValueInList \x {#6}  {\foreach \v in {2,4,6}  {
  \fill[black] ($ (\p\z)!\v/8! (\p\y) $) circle (0.5pt);
   }  }  
  {\draw (\o) node {}  -- (\p\x)node{}  ;}  
}  
}

\DeclareDocumentCommand\star{o m m m m g g}  {
\def \o {#2}    \def \n {#3}    \def \deg {#4}    \def \len {#5}  
\IfNoValueTF{#1}  {\def \rot {0}  }  {\def \rot {#1}   }  
\IfNoValueTF{#7}  {\def \p {p}  }  {\def \p {#7}   }  
\foreach \x in {0,1,...,\n}   {\coordinate[rotate around={-0.5*\deg+\rot: (\o)}  ](\p\x) at ($(\o)+({\deg * \x/\n}  : \len)$);}  
\foreach \x[evaluate=\x as \y using int(\x+1),evaluate=\x as \z using int(\x-1)] in {0,1,...,\n}  {

  {\draw (\o) node {}  -- (\p\x)node{}  ;}  
}  
}

\DeclareDocumentCommand\fitellipsis{o m m m m}  {
\IfNoValueTF{#1}  {\def \fc {gray}  }  {\def \fc {#1}  }  
\draw[fill=\fc!25, fill opacity=0.5] let \p1=(#2), \p2=(#3), \n1={atan2(\y2-\y1,\x2-\x1)},  \n2={veclen(\y2-\y1,\x2-\x1)}   in ($ (\p1)!0.5!(\p2) $) ellipse [radius=\n2/2+#4pt, y radius=#5pt, rotate=\n1];
\foreach \v in {4,5,6}   {
\fill ($ (#2)!\v/10!(#3) $) circle (0.5pt);
}  
}

\newcommand{\dotellipsis}  [4] 
{\draw[dashed] let \p1=(#1), \p2=(#2), \n1={atan2(\y2-\y1,\x2-\x1)},  \n2={veclen(\y2-\y1,\x2-\x1)}  
    in ($ (\p1)!0.5!(\p2) $) ellipse [radius=\n2/2+#3pt, y radius=#4pt, rotate=\n1];
}

\DeclareDocumentCommand\hyperpath{o m m m }  {
\IfNoValueTF{#1}  {\def \fco {gray}  }  {\def \fco {#1}  }  
\pgfmathtruncatemacro{\len}  {#4}  
\foreach \v in {0,...,\len}   {
	\coordinate (p\v) at ($(#2)!\v/\len!(#3)$);
	 }  

\pgfmathtruncatemacro{\lena}  {\len-1}  
\foreach \cur  [count=\next from 1] in {0,...,\lena}   {
	\fitellipsis[\fco]{p\cur}  {p\next}  {8}  {6}
}  
\foreach \v in {0,...,\len}   {
   \node[shade,shading=ball,circle,ball color=red!90!white,minimum size=1.8pt,inner sep=0pt] at  (p\v) {}  ;
	 }  
}

\newcommand{\partellipsis}  [4]{
\tkzDefMidPoint(#1,#2) \tkzGetPoint{tikzC}  
\tkzCalcLength[cm](#1,tikzC)\tkzGetLength{dAC}  
\tkzFindSlopeAngle(#1,#2)\tkzGetAngle{tkzang}  
\draw[rotate around={\tkzang: (tikzC)}  ] ($(tikzC) + ({(\dAC cm + #3 pt)* cos(180-70)},  { #4 pt * sin(180-70)}  )$) arc [start angle=180-70, end angle=180+70, x radius=\dAC cm+ #3 pt, y radius= #4 pt];
\draw[rotate around={\tkzang: (tikzC)}  ] ($(tikzC) + ({(\dAC cm + #3 pt) * cos(-70)},  { #4 pt * sin(-70)}  )$) arc [start angle=-70, end angle=70, x radius=\dAC cm+ #3 pt, y radius= #4 pt];
\node[rotate around={\tkzang: (tikzC)}  ] at (tikzC) {$\cdots$}  ;
}  

\newtheorem{thm}  {Theorem}  [section]
\newtheorem{lem}  {Lemma}  [section]
\newtheorem{definition}  {Definition}  [section]
\newtheorem{pop}  {Proposition}  [section]
\newtheorem{proof*}  {proof}  [section]
\newtheorem{cor}  {Corollary}  [section]

\newtheorem{remark}  {Remark}  
  
\newtheorem{theorem}  {Theorem}  [section]

\counterwithin{equation}{section}

\newcommand{\x}  {{\ensuremath{ \boldsymbol{x}  }  }  }  
\newcommand{\y}  {{\ensuremath{ \mathbf{y}  }  }  }  
\newcommand{\z}  {{\ensuremath{ \mathbf{z}  }  }  }

\newcommand{\Dtil}{\widetilde{D}}
\newcommand{\Btil}{\widetilde{B}}
\newcommand{\Etil}{\widetilde{E}}
\newcommand{\BDtil}{\widetilde{BD}}

\newcommand{\bver}[1]{\fill (#1) circle (2pt);}

\begin{document}  
\begin{frontmatter}
\title{Uniform Hypergraphs with \texorpdfstring{\(\alpha\)}{alpha}-Spectral Radius at Most That of a Loose Cycle}

\author[1]{Hangxi Cha}
\ead{2433949@tongji.edu.cn}
\author[1]{Xiaoqi Liu}
\ead{2111152@tongji.edu.cn}
\author[1]{Haiying Shan\corref{cor1}}
\ead{shan_haiying@tongji.edu.cn}
\affiliation[1]{organization={School of Mathematical Sciences, Tongji University},
            city={Shanghai},
            postcode={200092},
            country={P.R. China}}
\cortext[cor1]{Corresponding author}
\nonumnote{This work is partially supported by the National Natural Science Foundation of China (No. 12271182).}

	\begin{abstract}  
Let \(k\geq 3\) and \(\alpha\in[0,1)\), and let \(\lambda_k(\alpha)\) denote the \(\alpha\)-spectral radius of a \(k\)-uniform loose cycle. Lu and Man classified all connected \(k\)-uniform hypergraphs with adjacency spectral radius at most \(\lambda_k(0)\). In this paper, we extend their classification to the \(\alpha\)-spectral setting. In particular, for every \(k\geq 3\) and \(0<\alpha<1\), we prove that \(\lambda_k(\alpha)\) is the smallest limit point of the \(\alpha\)-spectral radii of connected \(k\)-uniform hypergraphs. Moreover, we give a complete comparison between the \(\alpha\)-spectral radius of an arbitrary finite connected \(k\)-uniform hypergraph \(H\) and that of a loose cycle by determining the sign of
\(
\rho_\alpha(H)-\lambda_k(\alpha).
\)
As a consequence, for each fixed \(0<\alpha<1\), we classify all finite
connected \(k\)-uniform hypergraphs with \(\alpha\)-spectral radius at most
\(\lambda_k(\alpha)\).
	\end{abstract}  
\begin{keyword}
uniform hypergraph \sep \(\alpha\)-spectral radius \sep adjacency tensor \sep loose cycle \sep \((\alpha,\beta)\)-labeling method

\MSC[2020] 05C50 \sep 05C75
\end{keyword}
\end{frontmatter}
\section{Introduction}  
In 1970, Smith~\cite{Smith} classified the connected graphs whose adjacency spectral radius is at most \(2\). The graphs with spectral radius less than \(2\) are the simply laced Dynkin diagrams, whereas those with spectral radius equal to \(2\) are the extended simply laced Dynkin diagrams. Guo, Mohar, Gavrilyuk, and Munemasa~\cite{GUO20172616,GavrilyukMunemasa2023} classified the digraphs with Hermitian spectral radius below \(2\). Cvetkovi\'c, Doob, and Gutman~\cite{MR683990} began the determination of graphs with spectral radius between \(2\) and \(\sqrt{2+\sqrt{5}}\); Brouwer and Neumaier~\cite{BrouwerNeumaier1989} completed it. Wang et al.~\cite{MR4548954} investigated signed graphs whose spectral radius does not exceed \(\sqrt{2+\sqrt{5}}\). Wang et al.~\cite{WANG20081606} studied graphs with spectral radii close to \(\frac{3}{2}\sqrt{2}\). Woo and Neumaier~\cite{MR2365422} and Lan and Lu~\cite{MR3034539} further investigated connected graphs \(H\) with spectral radii between \(\sqrt{2+\sqrt{5}}\) and \(\frac{3}{2}\sqrt{2}\). Hoffman~\cite{MR347860} proved that \(2\) is the smallest limit point of the spectral radii of connected graphs. Wang et al.~\cite{Belardo2020GraphsWA} subsequently showed that the smallest limit point of the \(A_\alpha\)-spectral radii of connected graphs is also \(2\) and characterized the connected graphs whose \(A_\alpha\)-spectral radius is at most \(2\). For hypergraphs, Lu and Man~\cite{LuMan} classified the connected uniform hypergraphs with adjacency spectral radius at most \(4^{1/k}\), the adjacency spectral radius of a \(k\)-uniform loose cycle.

In graph theory, an internal path of a graph \(H\) is a sequence of vertices
\(u_1,\ldots,u_n\) such that the \(u_i\) are distinct, except possibly
\(u_1=u_n\), and
\[
d_H(u_1)\geq3,\qquad
d_H(u_2)=\cdots=d_H(u_{n-1})=2,\qquad
d_H(u_n)\geq3,
\]
with \(u_iu_{i+1}\in E(H)\) for every \(i\in[n-1]\).

A hypergraph \(H\) consists of a nonempty vertex set \(V(H)\) and an edge
set \(E(H)\), where every edge \(e\in E(H)\) is a nonempty subset of
\(V(H)\)~\cite{berge}. Its order is \(|V(H)|\). For an integer \(k\geq2\),
the hypergraph \(H\) is \(k\)-uniform if every edge has size \(k\).
Throughout this paper, unless otherwise stated, all hypergraphs are assumed
to be connected and uniform with at least one edge.
Classification statements and membership in the families below are understood
up to hypergraph isomorphism.

Let \(H=(V,E)\) be a simple graph. The \(k\)-th power hypergraph of \(H\)
is obtained by adding \(k-2\) new vertices to each edge of \(H\), with
distinct new vertices used for distinct edges. The \(k\)-th powers of an
ordinary path and cycle with \(m\) edges are called, respectively, a loose
path and a loose cycle of length \(m\), and are denoted by \(P_m^k\) and
\(C_m^k\). In addition, \(C_2^k\) denotes the hypergraph consisting of two
distinct hyperedges that intersect in exactly two vertices.
A \emph{pendent loose path}, or simply a \emph{pendent path}, meets the rest
of the hypergraph only in one end vertex, called its root; its length is its
number of edges.

For positive integers \(k\) and \(n\), a real tensor
\(\mathcal{T}=(t_{i_1\cdots i_k})\) of order \(k\) and dimension \(n\) is a
multidimensional array satisfying
\[
t_{i_1\cdots i_k}\in\mathbb{R}
\qquad\text{for all }i_1,\ldots,i_k\in[n]:=\{1,2,\ldots,n\}.
\]
The tensor \(\mathcal{T}\) is symmetric if \(t_{i_1\cdots i_k}\) is invariant
under every permutation of its indices.

A real symmetric tensor \(\mathcal{T}\) of order \(k\) and dimension \(n\)
defines the associated homogeneous polynomial of degree \(k\)
\[
F_{\mathcal{T}}(x)=\mathcal{T}x^k
=\sum_{i_1,\ldots,i_k=1}^{n}t_{i_1\cdots i_k}x_{i_1}\cdots x_{i_k}.
\]

For \(x\in\mathbb{R}^n\), the quantity \(\mathcal{T}x^k\) is real, while
\(\mathcal{T}x^{k-1}\in\mathbb{R}^n\) has \(i\)th component
\[
(\mathcal{T}x^{k-1})_i
=\sum_{i_2,\ldots,i_k=1}^nt_{ii_2\cdots i_k}x_{i_2}\cdots x_{i_k}.
\]

\begin{definition}[\cite{Qi2005,QI2013228}]
Let \(\mathcal{T}\) be a real tensor of order \(k\) and dimension \(n\).
A scalar \(\lambda\in\mathbb{C}\) is an eigenvalue of \(\mathcal{T}\), with
corresponding eigenvector \(0\ne x\in\mathbb{C}^n\), if
\[
\mathcal{T}x^{k-1}=\lambda x^{[k-1]},
\]
where \(x^{[k-1]}\in\mathbb{C}^n\) is defined by
\((x^{[k-1]})_i=x_i^{k-1}\).
\end{definition}

In 2013, Shao~\cite{SHAO20132350} defined the general product of two tensors,
which is useful for studying the spectra of nonnegative tensors.

\begin{definition}[\cite{SHAO20132350}]
Let \(\mathcal{A}\) and \(\mathcal{B}\) be tensors of dimension \(n\) and
orders \(m\geq2\) and \(k\geq1\), respectively. Their product
\(\mathcal{AB}\) is the tensor \(\mathcal{C}\) of order
\((m-1)(k-1)+1\) and dimension \(n\), defined by
\[
c_{i\alpha_1\cdots\alpha_{m-1}}
=\sum_{i_2,\ldots,i_m=1}^{n}
a_{ii_2\cdots i_m}b_{i_2\alpha_1}\cdots b_{i_m\alpha_{m-1}},
\]
where \(i\in[n]\) and
\(\alpha_1,\ldots,\alpha_{m-1}\in[n]^{k-1}\).
\end{definition}

The spectral radius of $\mathcal{T}$ is defined as $\rho(\mathcal{T})=\max \{|\lambda|:\lambda \ \text{is an eigenvalue of}\  \mathcal{T}\}$.

For \(k\geq3\), let \(H=(V(H),E(H))\) be a \(k\)-uniform hypergraph on
\(n\) vertices. The adjacency tensor~\cite{CooperDutle} of
\(H\) is the order-\(k\), dimension-\(n\) tensor
\(\mathcal{A}(H)=(a_{i_1\cdots i_k})\), where
\[
a_{i_1\cdots i_k}=\begin{cases}
\frac{1}{(k-1)!},\ &\text{if}\ \{i_1,\cdots,i_k\}\in E(H),\\
0,\ &\text{otherwise}.
\end{cases}
\]
Let \(\mathcal{D}(H)\) be the order-\(k\), dimension-\(n\) diagonal tensor
whose \(i\)th diagonal entry is the degree \(d_i\) of vertex \(i\).
Then \(\mathcal{L}(H)=\mathcal{D}(H)-\mathcal{A}(H)\) and
\(\mathcal{Q}(H)=\mathcal{D}(H)+\mathcal{A}(H)\) are, respectively, the
Laplacian and signless Laplacian tensors of \(H\). Both
\(\mathcal{A}(H)\) and \(\mathcal{Q}(H)\) are nonnegative and symmetric.

Motivated by Nikiforov's \(A_\alpha\)-matrix~\cite{Nikiforov},
Lin, Guo, and Zhou~\cite{MR4048055} considered the corresponding convex
combination of the degree and adjacency tensors:
\[
\mathcal{A}_{\alpha}(H)=\alpha\mathcal{D}(H)+(1-\alpha)\mathcal{A}(H),
\qquad 0\leq\alpha<1.
\]
This tensor is called the \(A_\alpha\)-tensor of \(H\). Further results on the
\(\alpha\)-spectral radius of uniform hypergraphs can be found
in~\cite{WangShanWang2020,GuoZhou2020,HouChangShi2020}.

The spectral radius of $\mathcal{A}_{\alpha}(H)$ is called the ${\alpha}$-spectral radius of $H$ and denoted by $\rho_{\alpha}(H)$. Then $\rho_0(H)$ is the spectral radius of $\mathcal{A}(H)$, which is called the adjacency spectral radius of $H$. Moreover, $2\rho_{\frac{1}{2}}(H)$ is the spectral radius of $\mathcal{Q}(H)$, which is called the signless Laplacian spectral radius of $H$.

For $k\geq 3$, let $H$ be a $k$-uniform hypergraph with $V(H)=[n]$ and $x$ be an $n$-dimensional column vector. Clearly,
\[
\mathcal{A}_{\alpha}(H)x^k
=\alpha\sum_{i\in V(H)}d_ix_i^k
+(1-\alpha)\sum_{e\in E(H)}kx^e,
\]
or equivalently, 
\begin{equation}\label{eq:rayleigh-poly}
    \mathcal{A}_{\alpha}(H)x^k
    =\sum_{e\in E(H)}
    \left(\alpha\sum_{i\in e}x_i^k+(1-\alpha)kx^e\right),
\end{equation}
and for any $i\in V(H)$, $$(\mathcal{A}_{\alpha}(H)x^{k-1})_i=\alpha d_ix_i^{k-1}+(1-\alpha)\sum_{e\in E_i(H)}x^{e\setminus \{i\}},$$
where $x^e=x_{i_1}x_{i_2}\cdots x_{i_k}$ for $e=\{i_1,i_2,\cdots,i_k\}\in E(H)$.

It follows from \eqref{eq:rayleigh-poly} that
\begin{pop}[{\cite[Theorem~1.2]{ShanWangWang}}]\label{subgraph}
Suppose that \(H\) is a uniform hypergraph and \(H'\) is a subhypergraph of
\(H\). Then \(\rho_\alpha(H')\leq\rho_\alpha(H)\). Moreover, if \(H\) is
connected and \(H'\) is a proper subhypergraph, then
\(\rho_\alpha(H')<\rho_\alpha(H)\).
\end{pop}

Lu and
Man~\cite{LuMan} established a uniform-hypergraph analogue. They classified all connected \(k\)-uniform
hypergraphs with adjacency spectral radius at most
\[
        \rho_0(C^k)=4^{1/k},
\]
where \(C^k\) denotes a $k$-uniform loose cycle. They also showed that the smallest limit point
of the spectral radii of connected $k$-uniform hypergraphs is $4^{1/k}$.

Shan et al.~\cite{ShanWangWang} introduced the $(\alpha,\beta)$-labeling method for computing the $\alpha$-spectral radius of $k$-uniform hypergraphs and determined the $\alpha$-spectral radius of a $k$-uniform loose cycle.

\begin{lem}[{\cite[Theorem~3.3 and Remark~1(1)]{ShanWangWang}}]\label{loose spec}
The \(\alpha\)-spectral radius of a \(k\)-uniform loose cycle \(C^k\), as
well as that of \(C_2^k\), is the largest root of the equation
\begin{equation}\label{eq:cycle}
    \left(\frac{x}{2}-\alpha\right)^2(x-\alpha)^{k-2}-(1-\alpha)^k=0.
\end{equation}

\end{lem}
Thus, $\rho_\alpha(C^k)$ is independent of the length of the loose cycle. In the rest of this paper, we write $\rho_\alpha(C^k)=\lambda_k(\alpha)$.

Lemma~\ref{loose spec} readily yields
\begin{equation}
 \lambda_k(0)=4^{1/k},\qquad
 \lambda_k(\alpha)>2\alpha,\qquad
 \lim_{\alpha\to1-}\lambda_k(\alpha)=2.\label{2alpha}
\end{equation}

Evaluating the Rayleigh quotient at the constant unit \(k\)-norm vector gives
the value \(k/(k-1)\) when substituted into~\eqref{eq:rayleigh-poly}.
Therefore,
\begin{equation*}
        \lambda_k(\alpha)\geq\frac{k}{k-1}>1.
\end{equation*} In particular, \begin{equation}\label{eq:beta-range}
        0<\lambda_k(\alpha)^{-k}<1,
\end{equation}

The main results of this paper are stated as follows. The first one is about
the smallest limit point for the $\alpha$-spectral radius of $k$-uniform hypergraphs.
\begin{thm}\label{thmlimit}
For every fixed \(\alpha\in(0,1)\) and integer \(k\geq3\), the smallest
limit point of the \(\alpha\)-spectral radii of connected \(k\)-uniform
hypergraphs is \(\lambda_k(\alpha)\).
\end{thm}

Next, in Theorem~\ref{thm:global-five-type}, we determine the sign of $\rho_\alpha(H)-\lambda_k(\alpha)$ for every finite connected $k$-uniform hypergraph $H$ and every $0<\alpha<1$.

\begin{thm}
\label{thm:global-five-type}
Let \(k\geq3\), and let \(H\) be a finite connected \(k\)-uniform
hypergraph with at least one edge. Exactly one of the five hypotheses in
items~(1)--(5) below holds:
\begin{enumerate}[(1)]
\item if \(H\) is a loose path, then \(\rho_\alpha(H)-\lambda_k(\alpha)<0\) for every \(0<\alpha<1\).

\item if \(H\) is a loose cycle or $C_2^k$,
then \(\rho_\alpha(H)-\lambda_k(\alpha)=0\) for every \(0<\alpha<1\).
\item if \(H\) is a member of the list in
Proposition~\ref{prop:luman-inputs}(1), other than a loose cycle or $C_2^k$, then 
\(\rho_\alpha(H)-\lambda_k(\alpha)>0\) for every \(0<\alpha<1\).
\item if \(H\) is a non-path member of the list in
Proposition~\ref{prop:luman-inputs}(2), then there is a
unique \(\alpha_1(H,k)\in(0,1)\) such that
\[
\begin{aligned}
 &\rho_\alpha(H)-\lambda_k(\alpha)<0
     &&(0\leq\alpha<\alpha_1(H,k)),\\
 &\rho_{\alpha_1(H,k)}(H)-\lambda_k(\alpha_1(H,k))=0,&\\
 &\rho_\alpha(H)-\lambda_k(\alpha)>0
     &&(\alpha_1(H,k)<\alpha<1).
\end{aligned}
\]
\item if \(\rho_0(H)>4^{1/k}\), then
\(\rho_\alpha(H)-\lambda_k(\alpha)>0\) for every \(0<\alpha<1\).
\end{enumerate}
\end{thm}

For a fixed parameter, we use the following family.
\begin{definition}\label{def:fixed-alpha-family}
Let \(k\geq3\) and \(0<\alpha<1\). Let
\(\mathcal B_{k,\alpha}\) consist of the non-path hypergraphs \(H\) in
Proposition~\ref{prop:luman-inputs}(2) for which
\(\alpha\leq\alpha_1(H,k)\), where \(\alpha_1(H,k)\) is the transition
value in Theorem~\ref{thm:global-five-type}(4).
\end{definition}

\begin{thm}\label{thm:fixed-alpha-classification}
Let \(k\geq3\), fix \(0<\alpha<1\), and let \(H\) be a finite connected
\(k\)-uniform hypergraph with at least one edge. Then
\[
 \rho_\alpha(H)\leq\lambda_k(\alpha)
 \quad\Longleftrightarrow\quad
 H\in\{P_m^k:m\geq1\}\cup\{C_m^k:m\geq2\}
 \cup\mathcal B_{k,\alpha}.
\]
\end{thm}

The rest of the paper is organized as follows. Section~2 recalls the necessary
preliminaries and the classification of Lu and Man. Section~3 determines the
smallest limit point of the \(\alpha\)-spectral radii of \(k\)-uniform
hypergraphs. Section~4 characterizes \(
\rho_\alpha(H)-\lambda_k(\alpha)
\) for every connected \(k\)-uniform hypergraph whose adjacency spectral radius does not exceed that of a loose cycle. Section~5 investigates connected \(k\)-uniform hypergraphs whose adjacency spectral radius is greater than \(4^{1/k}\).
Section~\ref{sec:fixed-alpha} gives an explicit description of
\(\mathcal B_{k,\alpha}\) and proves Theorem~\ref{thm:fixed-alpha-classification}.

\section{Preliminaries}\label{sec:preliminaries}
\subsection{Hypergraph families with adjacency spectral radius at most
\texorpdfstring{\(4^{1/k}\)}{the loose-cycle threshold}}\label{sec:structures}
A hypergraph is \emph{linear} if any two distinct edges meet in at most one vertex. Connectedness is understood in the vertex--edge walk sense.  The incidence graph \(\mathcal I(H)\) is the bipartite graph with vertex classes \(V(H)\) and \(E(H)\), in which \(v\) and \(e\) are adjacent exactly when \(v\in e\).  We call \(H\) a \emph{hypertree} if \(\mathcal I(H)\) is a tree. 

A vertex of degree \(1\) is called a \emph{leaf vertex}. A vertex of degree at least three is a \emph{branching vertex}. In this paper,
an edge containing at least three non-leaf vertices is a \emph{branching
edge}. For \(k\geq3\), we use \emph{internal path} as follows. Between two
branching edges, it is a maximal loose path
whose first and last edges are the two branching edges and which has no other
branching edge or branching vertex internally. Its internal path parameter
\(q\) is the number of edges strictly between the two terminal branching
edges. An internal path from a branching edge \(e\) to a branching vertex
\(v\) is defined analogously, with \(e\) as its first edge and \(v\) as its
terminal vertex; here \(q\) is the number of internal degree-two propagation
steps. Thus \(q=0\) means, respectively, that the two branching edges meet or
that \(v\in e\). An internal path between two branching vertices is measured
by its ordinary edge length rather than by a parameter \(q\).

If \(H\) is \(r\)-uniform and \(k\geq r\), its \emph{\(k\)-uniform extension} is obtained by adding \(k-r\) new leaf vertices to every edge, with distinct new vertices used for distinct edges.  The case \(k=r+1\) is a one-step extension.  The uniform extension of a hypertree is a hypertree.  A \(k\)-uniform hypergraph \(H=(V,E)\), where \(k\geq3\), is called \emph{reducible} if every edge \(e\in E\) contains at least one leaf vertex, and \emph{irreducible} otherwise. Every reducible linear \(k\)-uniform hypergraph with \(k\geq4\) is a one-step extension of a linear \((k-1)\)-uniform hypergraph. In the uniformity reductions used below, we stop when the uniformity reaches \(3\) or when the current hypergraph is irreducible.

The hypergraph families used in the classification are as follows.  Unless otherwise stated, pendent-path lengths are positive integers and internal path parameters are nonnegative integers.

\begin{itemize}
\item \(E_{a,b,c}^{k}\) is obtained by attaching three pendent loose paths
of lengths \(a,b,c\) to a common branching vertex.
\item \(S_4^{k}\) is the four-edge hyperstar obtained by attaching four
length-one pendent loose paths to a common branching vertex.
\item \(F_{a,b,c}^{k}\) is obtained by attaching pendent loose paths of
lengths \(a,b,c\) at three distinct vertices of one branching edge.
\item \(H_{a,b,c,d}^{k}\), defined for \(k\geq4\), is obtained by attaching
four pendent loose paths at four distinct vertices of one branching edge.
\item \(G_{a,b:q:c,d}^{k}\) has a central loose path with \(q+2\) edges.
Its two end edges are branching edges.  Two pendent loose paths of lengths
\(a,b\) are attached at the two attachment vertices of the first end edge,
and paths of lengths \(c,d\) are attached similarly at the other end.
\item \(\Dtil_n^{k}\), \(n\geq5\), consists of two branching vertices,
each carrying two length-one pendent paths, joined by an internal path of length
\(n-4\).
\item \(\Btil_n^{k}=G_{1,2:n-8:1,2}^{k}\), \(n\geq8\).
\item \(\BDtil_n^{k}\), \(n\geq6\), consists of a branching vertex with
two length-one pendent paths and a branching edge with pendent paths of
lengths \(1\) and \(2\), joined by an internal path with parameter
\(n-6\).  When
\(n=6\), this parameter is zero, meaning that the branching vertex lies on
the branching edge.
\item \(BD_n^{k}\), \(n\geq5\), is defined in the same way as
\(\BDtil_n^{k}\), except that both pendent paths at the branching edge have
length one and the internal path parameter is \(n-5\).
\end{itemize}

\begin{figure}[htbp]
\centering
\newcommand{\brokenhyperpath}[2]{%
    \coordinate (armstart) at ($(#1)!0.37!(#2)$);
    \coordinate (armend) at ($(#1)!0.63!(#2)$);
    \hyperpath{#1}{armstart}{2}
    {\footnotesize\partellipsis{armstart}{armend}{8}{6}}
    \hyperpath{armend}{#2}{2}
}
	\begin{tikzpicture}[thick,scale=0.60]
		\begin{scope}[shift={(-9,0)}]
			\node at(0,6.5){\large$E_{a,b,c}^{k}$};
			\bver{0,0}
			\brokenhyperpath{0,0}{0,5.4}
			\brokenhyperpath{0,0}{-4.5,-3.9}
			\brokenhyperpath{0,0}{4.5,-3.9}
            \draw [decoration={brace,amplitude=6pt,raise=6pt},  decorate] (0,0) --  (0,5.4) node [midway, shift={(-18pt,1pt )}, scale=0.8]{$P_a$}  ;
            \draw [decoration={brace,amplitude=6pt,raise=6pt},  decorate] (0,0) --  (-4.5,-3.9) node [midway, shift={(14pt,-15pt )}, scale=0.8]{$P_b$}  ;
            \draw [decoration={brace,amplitude=6pt,raise=6pt},  decorate] (0,0) --  (4.5,-3.9) node [midway, shift={(18pt,12pt )}, scale=0.8]{$P_c$}  ;
		\end{scope}
		\begin{scope}[shift={(0,3)}]
			\node at(0,3.5){\large$S_{4}^{k}$};
			\bver{0,0}
			\hyperpath{0,0}{0,1.5}{1}
			\hyperpath{0,0}{-1.5,0}{1}
			\hyperpath{0,0}{0,-1.5}{1}
			\hyperpath{0,0}{1.5,0}{1}
		\end{scope}
		\begin{scope}[shift={(8,0)}]
		\node at(0,6.5){\large$F_{a,b,c}^{k}$};
	      \coordinate (v1) at (-0.6,0);
		\coordinate (v2) at (0,0);
		\coordinate (v3) at (0.6,0);	
		\fitellipsis[gray]{v1}{v3}{3}{5}
		\node[circle,fill=red!80,inner sep=0.8pt] at (v1) {};
		\node[circle,fill=red!80,inner sep=0.8pt] at (v2) {};
		\node[circle,fill=red!80,inner sep=0.8pt] at (v3) {};
		\coordinate (a) at (-2.7,4.5);
		\brokenhyperpath{v1}{a}
		\coordinate (b) at (0,-4.8);
		\brokenhyperpath{v2}{b}
		\coordinate (c) at (2.7,4.5);
		\brokenhyperpath{v3}{c}
        \draw [decoration={brace,amplitude=6pt,raise=6pt},  decorate] (-0.6,0) --  (-2.7,4.5) node [midway, shift={(-19pt,-2pt )}, scale=0.8]{$P_a$}  ;
        \draw [decoration=
        {brace,amplitude=6pt,raise=6pt},  decorate] (0,0) --  (0,-4.8) node [midway, shift={(16pt,0pt )}, scale=0.8]{$P_b$}  ;
        \draw [decoration=
        {brace,amplitude=6pt,raise=6pt,mirror},  decorate] (0.6,0) --  (2.7,4.5) node [midway, shift={(18pt,-5pt )}, scale=0.8]{$P_c$}  ;
		\end{scope}
		\begin{scope}[shift={(-7,-13)}]
		\node at(0,6.5){\large$H_{a,b,c,d}^{k}$};
	      \coordinate (v1) at (-0.8,0);
		\coordinate (v2) at (-0.2,0);
        \coordinate (v3) at (0.2,0);
		\coordinate (v4) at (0.8,0);	
		\fitellipsis[gray]{v1}{v4}{3}{5}
		\node[circle,fill=red!80,inner sep=0.8pt] at (v1) {};
		\node[circle,fill=red!80,inner sep=0.8pt] at (v2) {};
		\node[circle,fill=red!80,inner sep=0.8pt] at (v3) {};
		\node[circle,fill=red!80,inner sep=0.8pt] at (v4) {};
		\coordinate (a) at (-2.3,4.5);
		\brokenhyperpath{v1}{a}
		\coordinate (b) at (-0.5,-4.8);
		\brokenhyperpath{v2}{b}
        \coordinate (c) at (0.5,4.8);
		\brokenhyperpath{v3}{c}
		\coordinate (d) at (2.3,4.5);
		\brokenhyperpath{v4}{d}
        \draw [decoration={brace,amplitude=6pt,raise=6pt},  decorate] (-0.8,0) --  (-2.3,4.5) node [midway, shift={(-19pt,-2pt )}, scale=0.8]{$P_a$}  ;
        \draw [decoration=
        {brace,amplitude=6pt,raise=6pt},  decorate] (-0.2,0) --  (-0.5,-4.8) node [midway, shift={(16pt,0pt )}, scale=0.8]{$P_b$}  ;
        \draw [decoration=
        {brace,amplitude=6pt,raise=6pt},  decorate] (0.2,0) --  (0.5,4.8) node [midway, shift={(-16pt,0pt )}, scale=0.8]{$P_c$}  ;
        \draw [decoration=
        {brace,amplitude=6pt,raise=6pt,mirror},  decorate] (0.8,0) --  (2.3,4.5) node [midway, shift={(18pt,-5pt )}, scale=0.8]{$P_d$}  ;
		\end{scope}
		\begin{scope}[shift={(6.5,-13)}]
        \coordinate (A) at (-3,0) ;
        \coordinate (B) at (2.4,0);
        \coordinate (C) at (-1,0);
        \coordinate (D) at (0.4,0);
        \coordinate (A1) at (-2.6,0) ;
        \coordinate (B1) at (2,0);
        \hyperpath{A}  {C}  {2}  
        \partellipsis{C}  {D}  {8}  {6}
        \hyperpath{B}  {D}  {2}  
        \node[circle,fill=red!80,inner sep=0.8pt] at (A1) {};
        \node[circle,fill=red!80,inner sep=0.8pt] at (B1) {};
        \brokenhyperpath{A}{-4.8,4.5}
        \brokenhyperpath{B}{4.2,4.5}
        \brokenhyperpath{A1}{-1.1,4.5}
        \brokenhyperpath{B1}{0.5,4.5}
		\node at(0,6.5){\large$G_{a,b:q:c,d}^{k}$};
        \draw [decoration={brace,amplitude=6pt,raise=6pt,mirror},  decorate] (-3,0) --  (2.6,0) node [midway, shift={(1pt,-20pt )}, scale=0.8]{$P_{q+2}$}  ;
        \draw [decoration=
        {brace,amplitude=6pt,raise=6pt},  decorate] (-3,0) --  (-4.8,4.5) node [midway, shift={(-15pt,-7pt )}, scale=0.8]{$P_a$}  ;
        \draw [decoration=
        {brace,amplitude=6pt,raise=6pt},  decorate] (-2.6,0) --  (-1.1,4.5) node [midway, shift={(-17pt,6pt )}, scale=0.8]{$P_b$}  ;
        \draw [decoration=
        {brace,amplitude=6pt,raise=6pt},  decorate] (B) --  (4.2,4.5) node [midway, shift={(-17pt,6pt )}, scale=0.8]{$P_c$}  ;
        \draw [decoration=
        {brace,amplitude=6pt,raise=6pt},  decorate] (B1) --  (0.5,4.5) node [midway, shift={(-20pt,-7pt )}, scale=0.8]{$P_d$}  ;

		\end{scope}
		\begin{scope}[shift={(-7,-23)}]
		\node at(-0.3,3){\large$\widetilde{D}_{n}^{k}$};
        \coordinate (A) at (-3,0.5) ;
        \coordinate (B) at (2.4,0.5);
        \coordinate (C) at (-1,0.5);
        \coordinate (D) at (0.4,0.5);
        \coordinate (A1) at (-2.6,0.5) ;
        \coordinate (B1) at (2,0.5);
        \hyperpath{A}  {C}  {2}  
        \partellipsis{C}  {D}  {8}  {6}
        \hyperpath{B}  {D}  {2}  
        \hyperpath{A}{-3.7,1.3}{1}
        \hyperpath{B}{3.1,1.3}{1}
        \hyperpath{A}{-3.7,-0.3}{1}
        \hyperpath{B}{3.1,-0.3}{1}
        \draw [decoration={brace,amplitude=6pt,raise=6pt,mirror},  decorate] (-3,0.5) --  (2.6,0.5) node [midway, shift={(1pt,-20pt )}, scale=0.8]{$P_{n-4}$}  ;

		\end{scope}
		\begin{scope}[shift={(6.5,-23)}]
        \coordinate (A) at (-3,0.5) ;
        \coordinate (B) at (2.4,0.5);
        \coordinate (C) at (-1,0.5);
        \coordinate (D) at (0.4,0.5);
        \coordinate (A1) at (-2.6,0.5) ;
        \coordinate (B1) at (2,0.5);
        \hyperpath{A}  {C}  {2}  
        \partellipsis{C}  {D}  {8}  {6}
        \hyperpath{B}  {D}  {2}  
        \node[circle,fill=red!80,inner sep=0.8pt] at (B1) {};
        \hyperpath{A}{-3.7,1.3}{1}
        \hyperpath{A}{-3.7,-0.3}{1}
        \hyperpath{B}{3.9,2}{2}
        \hyperpath{B1}{1.3,1.3}{1}
		\node at(0,3){\large $\widetilde{BD}_n^k$};
        \draw [decoration={brace,amplitude=6pt,raise=6pt,mirror},  decorate] (-3,0.5) --  (2.6,0.5) node [midway, shift={(1pt,-20pt )}, scale=0.8]{$P_{n-5}$}  ;
		\end{scope}

	\end{tikzpicture}
\caption{The hypergraph families defined above.}
\label{fig:templates}
\end{figure}
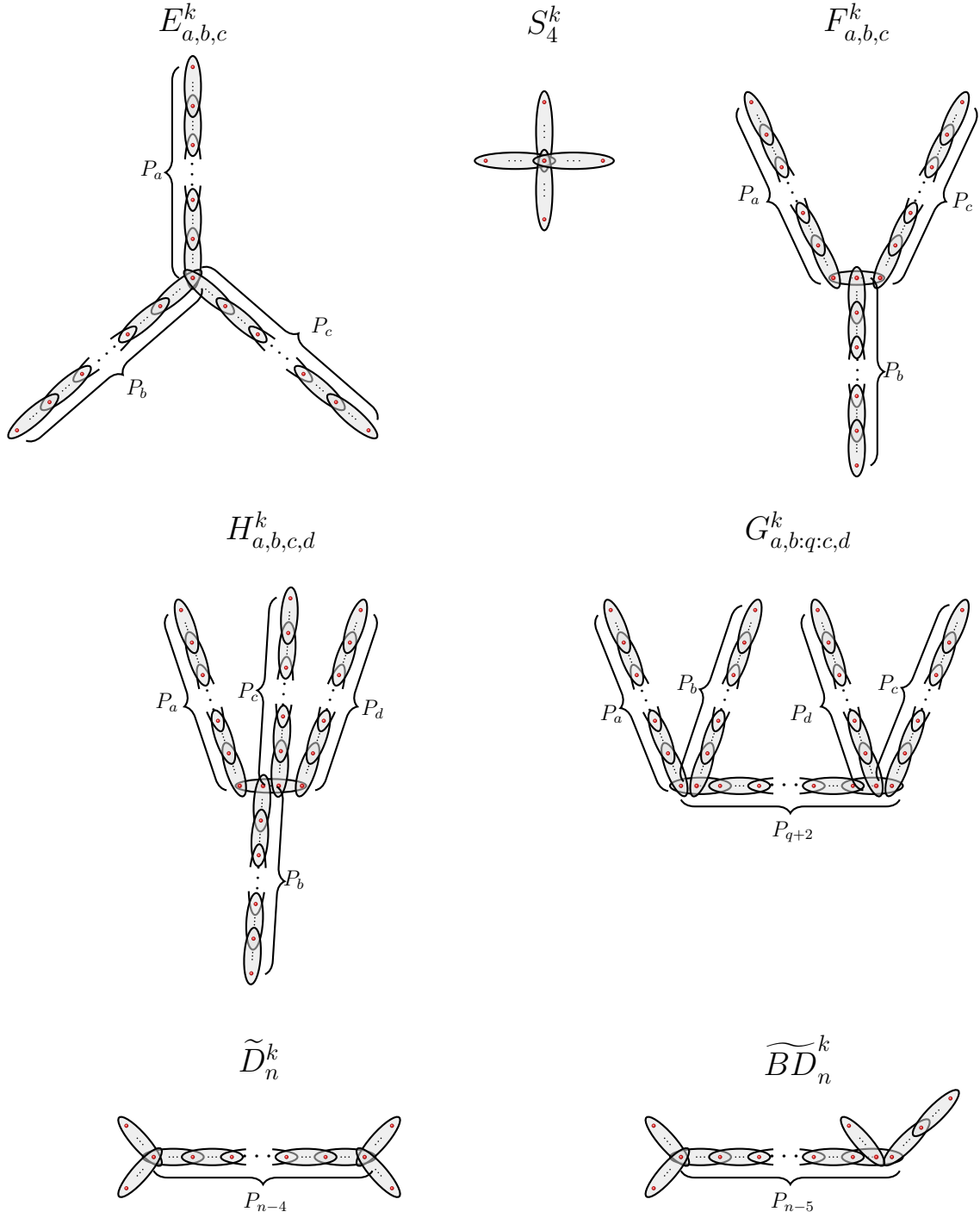

For later reference, we use the following identities:
\begin{equation*}
\begin{gathered}
 E_6=E_{1,2,2},\qquad E_7=E_{1,2,3},\qquad E_8=E_{1,2,4},\qquad
 \Etil_6=E_{2,2,2},\\\qquad \Etil_7=E_{1,3,3},\qquad
 \Etil_8=E_{1,2,5},\qquad D_n=E_{1,1,n-2},
\end{gathered}
\end{equation*}
and
\begingroup\footnotesize
\begin{equation*}
\begin{gathered}
 D'_n=F_{1,1,n-3},\qquad B_n=F_{1,2,n-4},\qquad
 B'_n=G_{1,1:n-6:1,1},\qquad
 \bar B_n=G_{1,1:n-7:1,2},\qquad
 \Btil_n=G_{1,2:n-8:1,2}.
\end{gathered}
\end{equation*}
\endgroup
The superscript \({}^{k}\) will be restored whenever the uniformity is material.

\Needspace{8\baselineskip}
\begin{pop}[{\cite[Theorems~1--2, Corollary~2, and Theorems~4--5]{LuMan}}]\label{prop:luman-inputs}
\noindent(1) The connected hypergraphs with adjacency spectral radius
\(4^{1/k}\) are the loose cycles, \(C_2^{k}\), the three infinite families
\[
 \Dtil_n^{k}\ (n\geq5),\qquad
 \Btil_n^{k}\ (n\geq8),\qquad
 \BDtil_n^{k}\ (n\geq6),
\]
and the exceptional hypergraphs
\[
\begin{gathered}
 S_4^{k},\ \Etil_6^{k},\ \Etil_7^{k},\ \Etil_8^{k},
 F_{2,3,4}^{k},\ F_{2,2,7}^{k},\ F_{1,5,6}^{k},\\
 F_{1,4,8}^{k},\ F_{1,3,14}^{k},\
 G_{1,1:0:1,4}^{k},\ G_{1,1:6:1,3}^{k}.
\end{gathered}
\] for $k\geq3$, and \(H_{1,1,2,2}^{k}\) for $k\geq 4$.

\noindent(2) The connected hypergraphs with adjacency spectral radius
strictly below \(4^{1/k}\) are the paths \(P_n^{k}\), the
families
\[
\begin{gathered}
 D_n^{k}=E_{1,1,n-2}^{k}\ (n\geq3),\qquad
 E_6^{k},E_7^{k},E_8^{k},\\
 BD_n^{k}\ (n\geq5),\qquad
 D_n'^{k}=F_{1,1,n-3}^{k}\ (n\geq4),\\
 B_n^{k}=F_{1,2,n-4}^{k}\ (n\geq5),\qquad
 B_n'^{k}=G_{1,1:n-6:1,1}^{k}\ (n\geq6),\\
 \bar B_n^{k}=G_{1,1:n-7:1,2}^{k}\ (n\geq7),
\end{gathered}
\]
and the exceptional hypergraphs
\[
\begin{gathered}
 F_{2,3,3}^{k},\qquad F_{2,2,j}^{k}\ (2\leq j\leq6),\qquad
 F_{1,3,j}^{k}\ (3\leq j\leq13),\\
 F_{1,4,j}^{k}\ (4\leq j\leq7),\qquad F_{1,5,5}^{k},\qquad
 G_{1,1:q:1,3}^{k}\ (0\leq q\leq5).
\end{gathered}
\] for \(k\geq3\), and \(H_{1,1,1,\ell}^{k}\ (1\leq\ell\leq4)\) for \(k\geq4\).
\end{pop}

\subsection{\texorpdfstring{$(\alpha,\beta)$}{(alpha,beta)}-labeling method}
Let \(H\) be a connected hypergraph, let \(0\leq\alpha<1\), and let \(0<\beta\leq1\).
Write \(\rho=\beta^{-1/k}\).  A weighted incidence matrix is a matrix
\(B=(B(v,e))\) such that \(B(v,e)>0\) when \(v\in e\) and \(B(v,e)=0\)
otherwise.  Throughout, we require \(B(v,e)>\alpha/\rho\) on every
incidence. Such a matrix is \((\alpha,\beta)\)-normal if
\begin{equation*}
 \sum_{e\ni v}B(v,e)=1,
 \qquad
 \prod_{v\in e}\frac{B(v,e)-\alpha/\rho}{1-\alpha}=\beta
\end{equation*}
for every vertex and edge.  It is subnormal if the vertex equalities are
replaced by \(\leq\) and the edge equalities by \(\geq\), and supernormal if
the inequalities are reversed.  A subnormal or supernormal labeling is
\emph{strict} if at least one of its defining inequalities is strict.

For a Berge cycle $v_0e_1v_1e_2\ldots e_s,v_s=v_0$, the labeling is \emph{consistent} if
\begin{equation*}
 \prod_{i=1}^s
 \frac{B(v_i,e_i)-\alpha/\rho}
      {B(v_{i-1},e_i)-\alpha/\rho}=1.
\end{equation*}
A labeling is consistent if this holds for every Berge cycle.  A hypertree has no
such cycle, so consistency is automatic.

\begin{lem}[{\cite[Definitions~3.1--3.3, Lemmas~3.1--3.2, and Theorem~3.1]{ShanWangWang}}]\label{lem:label-comparison}
Let \(H\) be a connected \(k\)-uniform hypergraph. 
\begin{enumerate}[(1)]
\item An \((\alpha,\beta)\)-subnormal labeling implies
\(\rho_\alpha(H)\leq\beta^{-1/k}\). Moreover, if the labeling is strict,
then \(\rho_\alpha(H)<\beta^{-1/k}\).
\item A consistently strictly \((\alpha,\beta)\)-supernormal labeling
implies \(\rho_\alpha(H)>\beta^{-1/k}\).
\item A consistent \((\alpha,\beta)\)-normal labeling implies
\(\rho_\alpha(H)=\beta^{-1/k}\).
\end{enumerate}
\end{lem}

\section{The smallest limit point of the
\texorpdfstring{\(\alpha\)}{alpha}-spectral radii of
\texorpdfstring{\(k\)}{k}-uniform hypergraphs}
We first prove that the $\alpha$-spectral radius of any nonlinear $k$-uniform hypergraph is at least $\lambda_k(\alpha)$.

\begin{lem}\label{sc:lem:cycle-subgraphs}
If a connected \(k\)-uniform hypergraph \(H\) contains a loose cycle
\(C_m^{k}\), \(m\geq3\), or \(C_2^{k}\) as a proper subhypergraph, then
\[
        \rho_\alpha(H)>\lambda_k(\alpha)\qquad(0<\alpha<1).
\]
\end{lem}

\begin{proof}
It immediately follows from Proposition \ref{subgraph} and Lemma \ref{loose spec}.
\end{proof}

\begin{lem}\label{sc:lem:large-intersection}
Let \(J\) be the hypergraph formed by two \(k\)-edges whose intersection has
size \(s\), where \(3\leq s\leq k-1\).  Then
\[
        \rho_\alpha(J)>\lambda_k(\alpha)\qquad(0<\alpha<1).
\]
Consequently every connected hypergraph containing two edges with an
intersection of size at least three satisfies the same strict inequality.
\end{lem}

\begin{proof}
Let \(\mu=\rho_\alpha(J)\).  If the common vertices have Perron
coordinate \(a\) and the private vertices coordinate \(b\), then
\[
        (\mu-2\alpha)a^{k-1}=2(1-\alpha)a^{s-1}b^{k-s},
        \qquad
        (\mu-\alpha)b^{k-1}=(1-\alpha)a^sb^{k-s-1}.
\]
Eliminating \(a/b\) gives
\begin{equation}\label{sc:eq:large-intersection}
        \left(\frac{\mu-2\alpha}{2}\right)^s
        (\mu-\alpha)^{k-s}=(1-\alpha)^k.
\end{equation}
At \(\mu=\lambda_k(\alpha)\), the left-hand side
of~\eqref{sc:eq:large-intersection} is
\[
        (1-\alpha)^k
        \left(\frac{\lambda_k(\alpha)-2\alpha}
        {2(\lambda_k(\alpha)-\alpha)}\right)^{s-2},
\]
which is strictly smaller than \((1-\alpha)^k\).  Since the left-hand side
of~\eqref{sc:eq:large-intersection} is strictly increasing for
\(\mu>2\alpha\), its root is larger
than \(\lambda_k(\alpha)\). Proposition~\ref{subgraph} completes the proof.
\end{proof}

We show that $\lambda_k(\alpha)$ is the limit value of the $\alpha$-spectral radii of paths.
\begin{lem}\label{pnlimit}
Let \(P_n^k\) be a \(k\)-uniform loose path with \(n\) edges. Then, for every
\(k\geq3\),
\[
\lim_{n\to\infty}\rho_\alpha(P_n^k)=\lambda_k(\alpha).
\]
\end{lem}
\begin{proof}
Since \(P_n^k\) is a proper subhypergraph of \(C_{n+1}^k\),
Proposition~\ref{subgraph} gives
\[
\rho_\alpha(P_n^k)<\lambda_k(\alpha).
\]
    
By the symmetry of loose cycles and the uniqueness of Perron vectors, its positive Perron vector takes the same value \(a>0\) at all vertices of degree $2$ and the same value \(b>0\) at all vertices of degree $1$. Define \[
D:=a^k+(k-2)b^k,
\]
\[
N:=
\alpha\bigl(2a^k+(k-2)b^k\bigr)
+
(1-\alpha)k a^2b^{k-2}.
\]
Substituting this Perron vector into the Rayleigh quotient for \(C^{k}\)
gives
\begin{equation}
    \lambda_k(\alpha)=\frac ND.\label{ND1}
\end{equation}

Write the edges of \(P_n^k\) as
\(e_i=\{v_{i-1},v_i\}\cup L_i\), \(1\leq i\leq n\), where
\(|L_i|=k-2\) and the sets \(L_i\) are pairwise disjoint. Define \(y\) by
\[
y_{v_i}=a\quad(0\leq i\leq n),\qquad
y_w=b\quad(w\in L_i).
\]
Then
\begin{equation}
    \sum_{v\in V(P_n^{k})}y_v^k=(n+1)a^k+n(k-2)b^k=nD+a^k.\label{ND2}
\end{equation}

The coordinate product over each edge is \(a^{2}b^{k-2}\), and hence
\begin{equation}
    \mathcal{A}_{\alpha}\left(P_{n}^{k}\right)y^{k}=nN.\label{ND3}
\end{equation}
Combining \eqref{ND1}, \eqref{ND2}, and \eqref{ND3}, we obtain
\[\rho_{\alpha}\left(P_{n}^{k}\right) \geq \frac{n N}{n D+a^{k}}=\lambda_{k}(\alpha) \cdot \frac{n D}{n D+a^{k}}. \]
Thus
\[
\lambda_{k}(\alpha)\frac{nD}{nD+a^{k}}
\leq \rho_{\alpha}\left(P_{n}^{k}\right)<\lambda_{k}(\alpha).
\]
Letting \(n\to\infty\) proves the result. Furthermore, from Proposition \ref{subgraph}, and 
$P_n^k\subsetneq P_{n+1}^k,$ we have 
$\rho_\alpha(P_n^k)<\rho_\alpha(P_{n+1}^k).
$
\end{proof}
\begin{lem}[{\cite[Theorem~5.1(1)]{ShanWangWang}}]\label{treelimit}
For \(k\geq3\) and \(n\geq4\), \(P_n^k\) is the unique hypergraph with
minimum \(\alpha\)-spectral radius among connected \(k\)-uniform hypergraphs
with \(n\) edges.
\end{lem}

Now we are ready to present the proof of Theorem \ref{thmlimit}.
\begin{proof}[\textit{Proof of Theorem~\ref{thmlimit}.}]
Let \(\alpha\in(0,1)\), and let \(H_1,H_2,\ldots\) be a sequence of connected
\(k\)-uniform hypergraphs such that
\(\rho_\alpha(H_i)\ne\rho_\alpha(H_j)\) for \(i\ne j\) and
\(\rho_\alpha(H_n)\to\lambda<\lambda_k(\alpha)\). For every fixed \(M\),
there are only finitely many \(k\)-uniform hypergraphs of order at most \(M\),
up to isomorphism. Because the \(H_i\) have pairwise distinct spectral radii,
it follows that \(|V(H_n)|\to\infty\). Moreover,
\(|V(H_n)|\leq k|E(H_n)|\), and hence \(|E(H_n)|\to\infty\).
By Lemma~\ref{treelimit},
\[
\rho_\alpha(H_n)\geq
\rho_\alpha\!\left(P^k_{|E(H_n)|}\right).
\]
Lemma~\ref{pnlimit} now contradicts
\(\rho_\alpha(H_n)\to\lambda<\lambda_k(\alpha)\). Therefore,
\(\lambda_k(\alpha)\) is the smallest limit point of the
\(\alpha\)-spectral radii of connected \(k\)-uniform hypergraphs.
\end{proof}
\section{Uniform hypergraphs whose adjacency spectral radius does not exceed that of a loose cycle}
\subsection{Hypergraphs whose
\texorpdfstring{\(\alpha\)}{alpha}-spectral radius never exceeds that of a loose cycle}
Every ordinary cycle has $\alpha$-spectral radius \(2\), independently of
its length, because it is \(2\)-regular. We next prove that the only connected
\(k\)-uniform hypergraphs whose \(\alpha\)-spectral radius never exceeds that
of a loose cycle are \(k\)-uniform loose paths, the \(k\)-uniform loose
cycles themselves, and \(C_2^k\).

For any edge $e\in E(H)$, define
\[
s_H(e) := \bigl|\{v\in e \mid d_H(v)=2\}\bigr|.
\]

\begin{thm}\label{neverexceed}
Let $k\ge 3$, and let $H$ be a connected $k$-uniform hypergraph. Then
\[
\rho_\alpha(H) \le \lambda_k(\alpha) \quad \forall \alpha\in[0,1)
\]
if and only if
\begin{equation}
\Delta(H)\le 2,\quad s_H(e)\le 2 \quad \forall e\in E(H). \label{never1}
\end{equation}
Moreover, under condition~\eqref{never1}, for each $0\le \alpha<1$,
\begin{itemize}
    \item If every edge satisfies $s_H(e)=2$, then $\rho_\alpha(H)=\lambda_k(\alpha)$,
    \item If there exists at least one edge with $s_H(e)<2$, then $\rho_\alpha(H)<\lambda_k(\alpha)$.
\end{itemize}
\end{thm}

\begin{proof}
Assume first that~\eqref{never1} holds. Fix \(0\leq\alpha<1\), and put
\[
 \rho=\lambda_k(\alpha),\qquad \beta=\rho^{-k},\qquad
 a=\frac{\alpha}{\rho}.
\]
By~\eqref{2alpha}, \(0\leq a<1/2\), and~\eqref{eq:cycle} gives
\[
 \left(\frac12-a\right)^2(1-a)^{k-2}=(1-\alpha)^k\beta.
\]
Set \(B(v,e)=1/d_H(v)\) for \(v\in e\). Then
\(B(v,e)\geq1/2>a\) and \(\sum_{e\ni v}B(v,e)=1\) for every \(v\).
If \(r=s_H(e)\leq2\), then
\[
 \prod_{v\in e}\frac{B(v,e)-\alpha/\rho}{1-\alpha}
 =\frac{(\frac12-a)^r(1-a)^{k-r}}{(1-\alpha)^k}
 \geq\beta,
\]
with equality if and only if \(r=2\). Lemma~\ref{lem:label-comparison}
therefore gives \(\rho_\alpha(H)\leq\lambda_k(\alpha)\), with strict
inequality if some edge satisfies \(s_H(e)<2\). If \(s_H(e)=2\) for every
edge, then \(B\) is normal. Moreover, every vertex on a Berge cycle has degree two,
so the two incidence values on each edge of the cycle are both \(1/2\).
Hence \(B\) is consistent, and Lemma~\ref{lem:label-comparison} gives
\(\rho_\alpha(H)=\lambda_k(\alpha)\).

Conversely, suppose that
\[
 \rho_\alpha(H)\leq\lambda_k(\alpha)\qquad(0\leq\alpha<1).
\]
Since \(\mathcal A_\alpha(H)\to\mathcal D(H)\) as \(\alpha\to1^-\),
continuity and~\eqref{2alpha} give
\[
 \Delta(H)=\lim_{\alpha\to1^-}\rho_\alpha(H)\leq2.
\]
Suppose that an edge \(e\) contains \(r=s_H(e)\geq3\) vertices of degree
two. Write
\[
 e=\{u_1,\ldots,u_r,w_1,\ldots,w_{k-r}\},\qquad
 d_H(u_i)=2,\quad d_H(w_j)=1.
\]
Let \(\alpha=1-\varepsilon\). From~\eqref{eq:cycle} and
\(\lambda_k(\alpha)-\alpha>\alpha\), there is a constant \(M_k>0\) such
that, for all sufficiently small \(\varepsilon>0\),
\[
 \lambda_k(\alpha)-2\alpha
 =2\varepsilon^{k/2}\bigl(\lambda_k(\alpha)-\alpha\bigr)^{-(k-2)/2}
 \leq M_k\varepsilon^{k/2}.
\]
If \(r<k\), set \(t=\varepsilon^{1/r}\), take \(x_{u_i}=1\),
\(x_{w_j}=t\), and set all other coordinates equal to zero. Every edge
other than \(e\) then has zero adjacency-product term. The Rayleigh quotient
gives
\[
\begin{aligned}
 \rho_\alpha(H)-2\alpha
 &\geq
 \frac{\varepsilon k t^{k-r}-\alpha(k-r)t^k}
      {r+(k-r)t^k}\\
 &=\frac{\bigl(k-\alpha(k-r)\bigr)\varepsilon^{k/r}}
        {r+(k-r)\varepsilon^{k/r}}
 \geq\frac{r}{k}\varepsilon^{k/r}.
\end{aligned}
\]
Since \(r\geq3\), we have \(k/r<k/2\), a contradiction for sufficiently
small \(\varepsilon\). If \(r=k\), the characteristic vector of \(e\)
instead gives
\[
 \rho_\alpha(H)\geq2\alpha+\varepsilon,
\]
which is again a contradiction for sufficiently small \(\varepsilon\),
since \(k\geq3\). Therefore \(s_H(e)\leq2\) for every edge \(e\), as
required.
\end{proof}

From Theorem \ref{neverexceed} we can characterize all connected $k$-uniform hypergraphs whose $\alpha$-spectral radius never exceeds that of a $k$-uniform loose cycle.

\begin{thm}\label{neverexceedmain}
Let $k\ge 3$, and let $H$ be a connected $k$-uniform hypergraph. Then
\[
\rho_\alpha(H) \le \lambda_k(\alpha) \quad \forall \alpha\in[0,1)
\]
if and only if $H$ is a $k$-uniform loose path, a $k$-uniform loose cycle, or $C_2^k$.
\end{thm}

\begin{proof}
If $H$ is not linear, there exist two edges $e_1,e_2\in E(H)$ with $|e_1\cap e_2|\ge 2$. By Theorem \ref{neverexceed}, $|e_1\cap e_2|\le 2$, so $|e_1\cap e_2|=2$. If $H$ contains any other edge, then $C_2^k$ is a proper subhypergraph of $H$. We know
\[
\rho_\alpha(C_2^k) = \lambda_k(\alpha) \quad \forall \alpha\in[0,1).
\]
By Proposition \ref{subgraph},
\[
\rho_\alpha(H) > \rho_\alpha(C_2^k),
\]
a contradiction. Hence $H\cong C_2^k$.

Now assume that \(H\) is linear. Since \(s_H(e)\leq2\) for every edge
\(e\), each edge contains at least \(k-2\) vertices of degree \(1\).
Delete \(k-2\) degree-one vertices from each edge of \(H\) to obtain a simple
graph \(G\). Then \(H\) can be recovered by adding \(k-2\) degree-one
vertices to each edge of \(G\). The construction preserves connectedness,
and Theorem~\ref{neverexceed} gives
\(\Delta(G)\leq\Delta(H)\leq2\). Thus \(G\) is a connected simple graph of
maximum degree at most \(2\), so \(G\) is a path or a cycle.
Consequently, \(H\) is a \(k\)-uniform loose path or a \(k\)-uniform loose
cycle.
\end{proof}
For a \(k\)-uniform loose path \(P_n^k\), there is an edge
\(e\in E(P_n^k)\) with \(s_{P_n^k}(e)<2\). Theorem~\ref{neverexceed}
therefore yields the following corollary.
\begin{cor}
    Let $k\ge 3$, and let $H$ be a connected $k$-uniform hypergraph. Then
\[
\rho_\alpha(H)  < \lambda_k(\alpha) \quad \forall \alpha\in[0,1)
\]
if and only if $H$ is a $k$-uniform loose path.
\end{cor}
\begin{remark}
When $k=2$, the condition $s_H(e)\le 2$ holds automatically. In this case the graphs with $\alpha$-spectral radius bounded above by that of any cycle are exactly simple graphs with maximum degree at most $2$, i.e., paths and cycles.
\end{remark}

\subsection{Hypergraphs whose adjacency spectral radius equals that of a loose cycle}

Throughout this subsection, let \(0<\alpha<1\), and write
\(\lambda=\lambda_k(\alpha)\) and
\(\xi=\frac{\lambda-2\alpha}{\lambda-\alpha}\).

\begin{lem}\label{lem:Tcriterion}
Let \(H\) be a connected \(k\)-uniform hypertree.  Suppose that positive
numbers \(T(v,e)\) are assigned to all incidences and satisfy
\begin{equation}\label{eq:Tcriterion-vertex}
 \sum_{e\ni v}T(v,e)- R_d\geq 0,
 \qquad \text{where}\quad
 R_d=\frac{\lambda-d\alpha}{\lambda-\alpha}=(d-1)\xi-(d-2)
\end{equation}
for every vertex \(v\) of degree \(d\), and
\begin{equation}\label{eq:Tcriterion-edge}
        \frac{\xi^2}{4}-\prod_{v\in e}T(v,e)\geq0
\end{equation}
for every edge.  If at least one of these inequalities is strict, then
\[
        \rho_\alpha(H)>\lambda_k(\alpha).
\]
\end{lem}


\begin{proof}
Define $Q=\frac{\lambda-\alpha}{1-\alpha}$. Equation \eqref{eq:cycle} gives 
\begin{equation}\label{eq:rq}
        \xi^2Q^k=4.
\end{equation}

Define the associated weighted incidence matrix by
\[
 B(v,e)=\frac{\alpha}{\lambda}
        +\frac{\lambda-\alpha}{\lambda}T(v,e).
\]
Then~\eqref{eq:Tcriterion-vertex} gives \(\sum_{e\ni v}B(v,e)\geq1\),
and
\[
 \prod_{v\in e}\frac{B(v,e)-\alpha/\lambda}{1-\alpha}
 =\lambda^{-k}Q^k\prod_{v\in e}T(v,e)
 \leq\lambda^{-k}
\]
by~\eqref{eq:rq}.  Hence \(B\) is an
\((\alpha,\lambda^{-k})\)-supernormal labeling.  It is consistent because
\(H\) is a hypertree, and it is strict by hypothesis.  The result follows
from Lemma~\ref{lem:label-comparison} and~\eqref{eq:beta-range}.
\end{proof}

\begin{remark}
    The quantities \(T(v,e)\) in Lemma~\ref{lem:Tcriterion} are auxiliary variables, not entries of a weighted incidence matrix.  The corresponding incidence weights \(B(v,e)\) are defined in its proof.  For the auxiliary variables, a vertex or edge condition is called \emph{normal} when equality holds in~\eqref{eq:Tcriterion-vertex} or~\eqref{eq:Tcriterion-edge}, respectively.
\end{remark}

The following two path labelings will be used for the $\alpha$-spectral comparison between the hypergraphs in Proposition \ref{prop:luman-inputs}(1) and loose cycles.

\begin{lem}\label{lem:point-arms}
Let \(0<\xi<1\) and define
\begin{equation}\label{eq:p-recursion}
 p_1=\frac{\xi^2}{4},
 \qquad
 p_{\ell+1}=\frac{\xi^2}{4(\xi-p_\ell)}\quad(\ell\geq1).
\end{equation}
A pendent loose path of length \(\ell\), attached to a branching vertex, can
be assigned positive values \(T(v,e)\) so that every edge and every internal
vertex is normal in Lemma~\ref{lem:Tcriterion}, with value \(p_\ell\) at the
branching vertex.  Moreover,
\begin{equation}\label{eq:p-bounds}
        0<p_\ell<\frac{\xi}{2}\qquad(\ell\geq1).
\end{equation}
\end{lem}

\begin{proof}
Let the path edges be \(e_1,\ldots,e_\ell\), starting at the branching
vertex \(v\), and let \(x_i=e_i\cap e_{i+1}\) for
\(1\leq i<\ell\).  Put \(T(v,e_1)=p_\ell\), and for
\(1\leq i<\ell\), put
\[
 T(x_i,e_i)=\xi-p_{\ell-i},
 \qquad
 T(x_i,e_{i+1})=p_{\ell-i}.
\]
Set \(T(v,e)=1\) at every leaf incidence.  Every internal vertex has sum
\(\xi\).  The recurrence~\eqref{eq:p-recursion} gives edge product
\(\frac{\xi^2}4\) on each path edge, including the terminal edge.  Finally,
\(p_1<\frac{\xi}2\), and \(p_\ell<\frac{\xi}2\) implies
\[
        \xi-p_\ell>\frac{\xi}{2}>0,
        \qquad
        p_{\ell+1}<\frac{\xi^2}{4(\xi/2)}=\frac{\xi}{2},
\]
so the denominator in~\eqref{eq:p-recursion} is positive at every step
and~\eqref{eq:p-bounds} follows.
\end{proof}

The values needed later are
\begin{equation*}
 p_1=\frac{\xi^2}{4},\qquad
 p_2=\frac{\xi}{4-\xi},\qquad
 p_3=\frac{\xi(4-\xi)}{4(3-\xi)},\qquad
 p_5=\frac{\xi(8-3\xi)}{4(5-2\xi)}.
\end{equation*}

\begin{lem}\label{lem:pendent-paths-edge}
Let \(e\) be an edge whose non-leaf vertices are \(x_1,\ldots,x_s\), and
write
\[
        h_i=\frac{T(x_i,e)}{\xi}\qquad(1\leq i\leq s),
\]
with \(T(v,e)=1\) at every leaf incidence on \(e\).  Then the edge
condition in Lemma~\ref{lem:Tcriterion} is normal exactly when
\(\prod_{i=1}^s h_i\) is
\begin{equation*}
        \theta_s=\frac{\xi^{2-s}}4.
\end{equation*}
Now let \(e_0,e_1,\ldots,e_\ell\) be a pendent loose path attached at the
branching edge \(e_0\):
\(e_{i-1}\cap e_i=\{x_i\}\) for \(1\leq i\leq\ell\), and the other vertices on
the path edges are leaf vertices.  Define
\begin{equation}\label{eq:y-g-recursion}
 y_\ell=\frac{\xi}{4},
 \qquad
 y_i=\frac{1}{4(1-y_{i+1})}\ (i=\ell-1,\ldots,1),
 \qquad
 g_\ell=1-y_1.
\end{equation}
There is a positive assignment of the variables \(T(v,e)\) on the path,
together with the incidence \((x_1,e_0)\), such that every path edge
\(e_i\) with \(i\geq1\) and every
internal shared vertex \(x_i\) is normal, and the value at the root edge is
\[
        T(x_1,e_0)=\xi g_\ell.
\]
Explicitly,
\[
        T(x_i,e_{i-1})=\xi(1-y_i),\qquad
        T(x_i,e_i)=\xi y_i\qquad(1\leq i\leq\ell),
\]
and \(T(v,e)=1\) at every leaf incidence on the path.  All these auxiliary
values are positive.
\end{lem}

\begin{proof}
The product on \(e\) is \(\xi^s\prod_i h_i\), so the equality case
of~\eqref{eq:Tcriterion-edge} is exactly \(\prod_i h_i=\frac{\xi^{2-s}}4\).
For the pendent-path assignment, each shared vertex satisfies
\[
        T(x_i,e_{i-1})+T(x_i,e_i)=\xi.
\]
The terminal edge \(e_\ell\) is normal because \(T(x_\ell,e_\ell)=\xi y_\ell
=\frac{\xi^2}4\).  If \(1\leq i<\ell\), then the non-leaf incidences on \(e_i\)
are \(\xi y_i\) and \(\xi(1-y_{i+1})\), whose product is \(\frac{\xi^2}4\)
by~\eqref{eq:y-g-recursion}.  Thus every non-root edge of the pendent path is normal.
Finally \(0<y_\ell<\frac14\), and backward induction in~\eqref{eq:y-g-recursion}
gives \(0<y_i<\frac12\).  Hence all assigned incidences are positive, and the root
incidence has auxiliary value \(\xi(1-y_1)=\xi g_\ell\).
\end{proof}

The required explicit values are
\begin{equation*}
\begin{gathered}
 g_1=\frac{4-\xi}{4},\quad
 g_2=\frac{3-\xi}{4-\xi},\quad
 g_3=\frac{8-3\xi}{4(3-\xi)},\quad
 g_4=\frac{5-2\xi}{8-3\xi},\\
 g_5=\frac{12-5\xi}{4(5-2\xi)},\quad
 g_6=\frac{7-3\xi}{12-5\xi},\quad
 g_7=\frac{16-7\xi}{4(7-3\xi)},\\
 g_8=\frac{9-4\xi}{16-7\xi},
 \qquad
 g_{14}=\frac{15-7\xi}{28-13\xi}.
\end{gathered}
\end{equation*}

\begin{pop}\label{prop:infinite-threshold-labelings}
For every \(k\geq3\) and \(0<\alpha<1\), each of
\[
        \Dtil_n^{k}\ (n\geq5),\qquad
        \Btil_n^{k}\ (n\geq8),\qquad
        \BDtil_n^{k}\ (n\geq6)
\]
admits positive values \(T(v,e)\) satisfying
\eqref{eq:Tcriterion-vertex} and~\eqref{eq:Tcriterion-edge}, with at least
one of these inequalities strict.  More precisely:
\begin{enumerate}[(1)]
\item In \(\Dtil_n^{k}\), both branching-vertex inequalities are strict and all
other vertex and edge conditions are normal.
\item In \(\Btil_n^{k}\), both branching-edge inequalities are strict and all other
conditions are normal.
\item In \(\BDtil_n^{k}\), the branching-vertex and branching-edge
inequalities are both strict and all other conditions are normal.
\end{enumerate}
Consequently, if \(H\) is any of the listed hypergraphs, then \(H\) admits a
consistently strictly \((\alpha,\lambda_k(\alpha)^{-k})\)-supernormal
incidence labeling and satisfies
\(\rho_\alpha(H)>\lambda_k(\alpha)\) for $0<\alpha<1$.
\end{pop}

\begin{proof}
For \(\Dtil_n^{k}\), label each of the four length-one pendent paths by
Lemma~\ref{lem:point-arms}.  Set \(T(v,e)=\frac{\xi}{2}\) at every non-leaf incidence
on the loose path of length \(n-4\) joining the two branching vertices, and
set \(T(v,e)=1\) at every leaf incidence.  Every edge of the internal path
has product \(\frac{\xi^2}{4}\), and every internal vertex on that path has sum \(\xi\).
At each of
the two branching vertices the excess over \(R_3=2\xi-1\) is
\begin{equation}\label{eq:Dtil-gap}
 2p_1+\frac{\xi}{2}-R_3
 =\frac{(1-\xi)(2-\xi)}2>0.
\end{equation}
Thus both branching-vertex inequalities are strict.

For \(\Btil_n^{k}=G_{1,2:n-8:1,2}^{k}\), label all four exterior edge paths
normally by Lemma~\ref{lem:pendent-paths-edge}.  On the central internal path, assign
\(T=\frac\xi2\) to each non-leaf incidence on the internal path.  Then every internal
vertex on the internal path has sum \(\xi\), and every internal edge of that
path has product
\(\frac{\xi^2}4\).  At each branching edge, the three factors \(\frac{T}{\xi}\) have
product \(\frac12g_1g_2\), and
\begin{equation}\label{eq:Btil-gap}
 \theta_3-\frac12g_1g_2
 =\frac{1}{4\xi}-\frac12g_1g_2
 =\frac{(2-\xi)(1-\xi)}{8\xi}>0.
\end{equation}
Hence both branching-edge inequalities are strict.

Finally consider \(\BDtil_n^{k}\).  Label the two pendent paths at the
branching vertex by Lemma~\ref{lem:point-arms} and the two pendent paths at
the branching edge, of lengths \(1,2\), by
Lemma~\ref{lem:pendent-paths-edge}.  Set \(T(v,e)=\frac\xi2\) at every non-leaf
incidence on the internal path.  When the internal path parameter is zero,
this means assigning \(T=\frac\xi2\) to the incidence of the branching vertex on
the branching edge.  All internal path conditions are normal.  The branching
vertex has the positive gap~\eqref{eq:Dtil-gap}, while the branching edge has
the positive gap~\eqref{eq:Btil-gap}.  These are the only strict conditions.
Under uniform extension, set \(T(v,e)=1\) at each new leaf incidence, the same
assignments therefore work for every \(k\geq3\) obtained by the stated
uniform extension.
\end{proof}

\begin{pop}\label{prop:finite-threshold-labelings}
If $H$ is one of the following hypergraphs
\[
\begin{gathered}
 S_4^{k},\ \tilde{E}_6^{k},\ \tilde{E}_7^{k},\ \tilde{E}_8^{k},\ 
 F_{2,3,4}^{k},\ F_{2,2,7}^{k},\ F_{1,5,6}^{k},\
 F_{1,4,8}^{k},\ F_{1,3,14}^{k},\
 G_{1,1:0:1,4}^{k},\ G_{1,1:6:1,3}^{k},
\end{gathered}
\]
where $k\geq3$, or $H_{1,1,2,2}^{k}$ with $k\geq 4$, then \(H\) admits a
consistently strictly \((\alpha,\lambda_k(\alpha)^{-k})\)-supernormal
incidence labeling and satisfies
\(\rho_\alpha(H)>\lambda_k(\alpha)\) for $0<\alpha<1$.
\end{pop}
\begin{proof}
For
\(H\in\{S_4^{k},\tilde{E}_6^{k},\tilde{E}_7^{k},\tilde{E}_8^{k}\}\),
label each pendent path using Lemma~\ref{lem:point-arms}. The resulting
left-hand sides of the branching-vertex inequalities are listed in
Table~\ref{tab:point-branch-gaps}.

\begin{table}[!hbp]
\centering
\small
\caption{LHS of branching-vertex inequalities for $S_4^{k},\ \tilde{E}_6^{k},\ \tilde{E}_7^{k},\ \tilde{E}_8^{k}$}
\label{tab:point-branch-gaps}
\begin{tabular}{@{}ll@{}}
\toprule
Hypergraph &  LHS of branching-vertex inequality\\
\midrule
\(S_4\) &
\(\displaystyle 4p_1-R_4=(1-\xi)(2-\xi)\)\\[2mm]
\(\Etil_6=E_{2,2,2}\) &
\(\displaystyle 3p_2-R_3=\frac{2(1-\xi)(2-\xi)}{4-\xi}\)\\[3mm]
\(\Etil_7=E_{1,3,3}\) &
\(\displaystyle p_1+2p_3-R_3=
 \frac{(1-\xi)(2-\xi)(6-\xi)}{4(3-\xi)}\)\\[3mm]
\(\Etil_8=E_{1,2,5}\) &
\(\displaystyle p_1+p_2+p_5-R_3=
 \frac{(1-\xi)(2-\xi)(\xi^2-10\xi+20)}
 {2(4-\xi)(5-2\xi)}\)\\
\bottomrule
\end{tabular}
\end{table}

For $H\in \{F_{2,3,4}^{k},\ F_{2,2,7}^{k},\ F_{1,5,6}^{k}
,\ F_{1,4,8}^{k},\ F_{1,3,14}^{k},\ G_{1,1:0:1,4}^{k},\ G_{1,1:6:1,3}^{k},\ H_{1,1,2,2}^{k}\}$, label each pendent path using Lemma~\ref{lem:pendent-paths-edge}. For $G_{1,1:0:1,4}^{k}$, the two \(T/\xi\)-factors on the connecting
path are \(\frac49\) and \(\frac59\).  For \(G_{1,1:6:1,3}\), use the path
factors
\begin{equation*}
\begin{aligned}
 x_0&=\frac49, & x_1&=\frac{9}{20}, &
 x_2&=\frac{5}{11}, & x_3&=\frac{11}{24},\\
 x_4&=\frac{6}{13}, & x_5&=\frac{13}{28}, &
 x_6&=\frac{7}{15}, & x_7&=\frac{8}{15}.
\end{aligned}
\end{equation*}
Here the internal path is oriented from the left branching edge to the right branching edge. \(x_0\) and \(x_7\) are the \(T/\xi\)-factors
at the branching edges, and the intervening path incidences
are assigned the
corresponding complementary \(\frac{T}\xi\)-factors so that the \(i\)th internal
edge of the internal path has product \((1-x_{i-1})x_i\).  The identities
\[
        (1-x_{i-1})x_i=\frac14\qquad(1\leq i\leq6)
\]
make the six internal edges of the internal path normal. At each internal vertex on that path, the two incident \(\frac{T}\xi\)-factors are complementary. The left-hand sides of the branching-edge inequalities are listed in Table~\ref{tab:branch-edge-gaps}.

\begin{table}[!htbp]
\centering
\footnotesize
\caption{Left-hand sides of the branching-edge inequalities for the listed hypergraphs from Proposition~\ref{prop:luman-inputs}(1)}
\label{tab:branch-edge-gaps}
\begin{tabularx}{\textwidth}{@{}>{\raggedright\arraybackslash}p{0.28\textwidth}X@{}}
\toprule
Hypergraph &   LHS of branching-edge inequality\\
\midrule
\(F_{2,3,4}\) &
\(\displaystyle \xi^3(\frac{1}{4\xi}-g_2g_3g_4)=
 \frac{\xi^3(2-\xi)(1-\xi)}{2\xi(4-\xi)}\)\\[2mm]
\(F_{2,2,7}\) &
\(\displaystyle \xi^3(\frac{1}{4\xi}-g_2^2g_7)=
 \frac{\xi^3(2-\xi)(1-\xi)(7\xi^2-40\xi+56)}
 {4\xi(4-\xi)^2(7-3\xi)}\)\\[3mm]
\(F_{1,5,6}\) &
\(\displaystyle \xi^3(\frac{1}{4\xi}-g_1g_5g_6)=
 \frac{\xi^3(2-\xi)(1-\xi)(10-3\xi)}{16\xi(5-2\xi)}\)\\[3mm]
\(F_{1,4,8}\) &
\(\displaystyle \xi^3(\frac{1}{4\xi}-g_1g_4g_8)=
 \frac{\xi^3(2-\xi)(1-\xi)(4\xi^2-23\xi+32)}
 {2\xi(8-3\xi)(16-7\xi)}\)\\[3mm]
\(F_{1,3,14}\) &
\(\displaystyle \xi^3(\frac{1}{4\xi}-g_1g_3g_{14})=
 \frac{\xi^3(2-\xi)(1-\xi)(21\xi^2-122\xi+168)}
 {16\xi(3-\xi)(28-13\xi)}\)\\[3mm]
\addlinespace[1mm]
\(G_{1,1:0:1,4}\), left &
\(\displaystyle \xi^3(\frac{1}{4\xi}-\frac49g_1^2)=
 \frac{\xi^3(1-\xi)(\xi^2-7\xi+9)}{36\xi}\)\\[3mm]
\(G_{1,1:0:1,4}\), right &
\(\displaystyle \xi^3(\frac{1}{4\xi}-\frac59g_1g_4)=
 \frac{\xi^3(1-\xi)(10\xi^2-55\xi+72)}{36\xi(8-3\xi)}\)\\[3mm]
\(G_{1,1:6:1,3}\), left &
\(\displaystyle \xi^3(\frac{1}{4\xi}-\frac49g_1^2)=
 \frac{\xi^3(1-\xi)(\xi^2-7\xi+9)}{36\xi}\)\\[3mm]
\(G_{1,1:6:1,3}\), right &
\(\displaystyle \xi^3(\frac{1}{4\xi}-\frac{8}{15}g_1g_3)=
 \frac{\xi^3(1-\xi)(6\xi^2-34\xi+45)}{60\xi(3-\xi)}\)\\[3mm]
\(H_{1,1,2,2}\) &
\(\displaystyle \xi^4(\frac{1}{4\xi^2}-g_1^2g_2^2)=
 \frac{\xi^4(2-\xi)(1-\xi)(2+3\xi-\xi^2)}{16\xi^2}\)\\
\bottomrule
\end{tabularx}
\end{table}

It remains to show that the left-hand sides of the inequalities in
Tables~\ref{tab:point-branch-gaps} and~\ref{tab:branch-edge-gaps} are strictly
positive. Every displayed denominator is positive for \(0<\xi<1\). Each of
\[
\begin{gathered}
 \xi^2-10\xi+20,\quad 7\xi^2-40\xi+56,\quad
 4\xi^2-23\xi+32,\quad 21\xi^2-122\xi+168,\\
 \xi^2-7\xi+9,\quad 10\xi^2-55\xi+72,\quad
 6\xi^2-34\xi+45
\end{gathered}
\]
is decreasing on \([0,1]\) and positive at \(1\), while
\(2+3\xi-\xi^2>0\). Thus every listed left-hand side is strictly positive,
and Lemma~\ref{lem:Tcriterion} completes the proof.
\end{proof}

\begin{thm}\label{thm:equality-families}
Let \(k\geq3\), and let \(H\) be a finite connected \(k\)-uniform
hypergraph with \(\rho_0(H)\leq4^{1/k}\).  Then
\[
        \rho_\alpha(H)-\lambda_k(\alpha)>0
        \qquad \text{for every}\quad0<\alpha<1
\]
if and only if \(H\) belongs to the list in Proposition~\ref{prop:luman-inputs}(1) and is neither a
loose cycle nor \(C_2^{k}\).
\end{thm}

\begin{proof}
Assume first that \(\rho_\alpha(H)>\lambda_k(\alpha)\) for every
\(0<\alpha<1\).  If \(\rho_0(H)<4^{1/k}\), continuity would give
\(\rho_\alpha(H)<\lambda_k(\alpha)\) for all sufficiently small positive
\(\alpha\), a contradiction. Thus \(\rho_0(H)=4^{1/k}\), and \(H\)
belongs to the list in Proposition~\ref{prop:luman-inputs}(1).
Theorem~\ref{neverexceedmain} excludes the loose cycles and \(C_2^{k}\).
The converse follows from Propositions~\ref{prop:infinite-threshold-labelings}
and~\ref{prop:finite-threshold-labelings}.
\end{proof}

\subsection{Hypergraphs with adjacency spectral radius smaller than that of a loose cycle}
In this subsection, we investigate hypergraphs whose adjacency spectral radius
is smaller than that of a loose cycle. We prove that, unless \(H\) is a loose
path, there is a unique \(\alpha_1\in(0,1)\) such that
\[
\begin{aligned}
 &\rho_\alpha(H)-\lambda_k(\alpha)<0 &&(0\leq\alpha<\alpha_1),\\
 &\rho_{\alpha_1}(H)-\lambda_k(\alpha_1)=0,&\\
 &\rho_\alpha(H)-\lambda_k(\alpha)>0 &&(\alpha_1<\alpha<1).
\end{aligned}
\]

For comparison with a loose cycle, put \(\lambda=\lambda_k(\alpha)\).  The
normal edge equation for the cycle is
\begin{equation}\label{eq:cycle-normal}
 \left(\frac{1/2-\alpha/\lambda}{1-\alpha}\right)^2
 \left(\frac{1-\alpha/\lambda}{1-\alpha}\right)^{k-2}
 =\lambda^{-k}.
\end{equation}
This is exactly~\eqref{eq:cycle} in incidence-labeling form.

Fix \(0\leq\alpha<1\), let \(\lambda=\lambda_k(\alpha)\), and set
\begin{equation*}
 y=\frac{\alpha}{\lambda},\qquad
 t=\frac{1-2y}{2(1-y)}.
\end{equation*}
Then \(0<t\leq1/2\), and~\eqref{eq:cycle} yields
\begin{equation}\label{eq:psi}
 \alpha=\psi_k(t):=\frac{1-2t}{1-2t+t^{2/k}}.
\end{equation}
Moreover,
\begin{equation*}
 \psi_k'(t)=
 -\frac{2t^{2/k-1}(1+(k-2)t)}
 {k(1-2t+t^{2/k})^2}<0.
\end{equation*}
Thus \(t=1/2\) corresponds to \(\alpha=0\), while \(t\downarrow0\)
corresponds to \(\alpha\uparrow1\).

For a hypertree, assign weight \(1\) to every leaf incidence and write each
non-leaf incidence in the form
\begin{equation}\label{eq:normalized-incidence}
        B(v,e)=y+t(1-y)a(v,e).
\end{equation}
If an edge contains \(s\) non-leaf vertices with variables
\(a_1,\ldots,a_s\), then~\eqref{eq:cycle-normal} shows that its normal edge
condition is
\begin{equation}\label{eq:edge-normal-a}
        \prod_{i=1}^s a_i=t^{2-s}.
\end{equation}
Indeed, a non-leaf incidence supplies the factor
\[
 \frac{B(v,e)-y}{1-\alpha}
 =\frac{t(1-y)}{1-\alpha}a(v,e),
\]
whereas a leaf incidence with \(B(v,e)=1\) supplies the factor
\((1-y)/(1-\alpha)\). Thus an edge with \(s\) non-leaf incidences and
\(k-s\) leaf vertices has edge product
\[
 \lambda^{-k}Q^k t^s\prod_{i=1}^sa_i,\qquad
 Q=\frac{\lambda-\alpha}{1-\alpha}.
\]
The equation \eqref{eq:cycle} gives \(Q^k=t^{-2}\).  Hence the factors supplied
by the leaf vertices are already absorbed in this identity, and the same
equation \(\prod_i a_i=t^{2-s}\) remains valid after uniform extension.
At vertices of degree two and three, the normal vertex conditions become,
respectively,
\begin{equation*}
        a_1+a_2=2,
        \qquad
        a_1+a_2+a_3=4-\frac1t.
\end{equation*}

Along an ordinary loose path, propagation under the normality equations is governed by
\begin{equation*}
 f(x)=\frac{1}{2-x},\qquad
 f^m(x)=\frac{m-(m-1)x}{m+1-mx}\quad(m\geq0).
\end{equation*}
For \(0\leq x\leq1/2\),
\begin{equation}\label{eq:f-derivative}
        (f^m)'(x)=\frac{1}{(m+1-mx)^2}\leq1.
\end{equation}
A pendent path of length \(\ell\), propagated from its leaf end, has value
\begin{equation*}
        f^{\ell-1}(t)
\end{equation*}
at a branching vertex.  If it is attached to a branching edge, the value at
that edge is
\begin{equation*}
 u_\ell(t)=2-f^{\ell-1}(t)
 =\frac{\ell+1-\ell t}{\ell-(\ell-1)t}.
\end{equation*}
Finally,
\begin{equation}\label{eq:uell-log}
 0<-\frac{u_\ell'(t)}{u_\ell(t)}
 =\frac{1}{(\ell-(\ell-1)t)(\ell+1-\ell t)}
 \leq\frac23
 \qquad(0<t\leq1/2).
\end{equation}

\begin{lem}\label{lem:residual-sign}
Fix \(k\geq3\) and \(0<t\leq1/2\), put
\(\alpha=\psi_k(t)\) and \(\beta=\lambda_k(\alpha)^{-k}\), and use the
variables in~\eqref{eq:normalized-incidence}.  Assume that these
variables determine an incidence assignment on a \(k\)-uniform hypergraph
\(H\), and that every \(a(v,e)\) and every complementary
variable produced at a degree-two vertex is positive.  Suppose that all labeling
conditions are normal except possibly a collection of vertex and edge
conditions.  For a vertex condition write
\[
        R_v=S_v-U_v,
\]
where \(S_v=\frac{1-d_vy}{t(1-y)}\) and \(U_v=\sum_{e\ni v}a(v,e)\).  For
an edge condition with \(s\) non-leaf variables \(a_1,\ldots,a_s\), write
\[
        R_e=t^{s-2}\prod_{i=1}^sa_i-1.
\]
We call $R_v$ and $R_e$ the residuals.
\begin{enumerate}[(1)]
\item If all residuals are nonnegative and at least one is positive, the
induced \((\alpha,\beta)\)-labeling is strictly subnormal.
\item If all residuals are nonpositive and at least one is negative, the
induced \((\alpha,\beta)\)-labeling is strictly supernormal.  If it is
consistent, in particular if \(H\) is a hypertree, it is consistently strictly
supernormal.
\end{enumerate}
\end{lem}

\begin{proof}
At a vertex, substituting~\eqref{eq:normalized-incidence} shows that the sign
of \(U_v-S_v\) is the sign of \(\sum_{e\ni v}B(v,e)-1\).  At an edge,
\eqref{eq:edge-normal-a} is exactly the normal edge equation, and the
subnormal edge inequality is the direction in which the product of the \(a_i\)
is larger.  Thus nonnegative residuals give all subnormal inequalities, and
nonpositive residuals give all supernormal inequalities; a strict residual
gives strictness.  Consistency is inherited from the assignment, and is
automatic on a hypertree.
\end{proof}

\begin{lem}\label{lem:propagation-domain}
Let \(f(x)=1/(2-x)\).
\begin{enumerate}[(1)]
\item For an internal path from a branching edge to a branching vertex with
parameter \(q\), the value received by the
branching vertex is \(f^q(x)\).  If \(q=0\), the only requirement is
\(x>0\).  If \(q\geq1\), all intermediate incidences on the internal path
and their
degree-two complements are positive exactly on
\[
        0<x<\eta_q:=1+\frac1q.
\]
On this domain one has \(0<f^j(x)<2\) for \(0\leq j<q\) and
\(f^q(x)>0\).  The terminal value \(f^q(x)\) enters a branching vertex sum,
so it is not required to be less than \(2\).
\item For an internal path between two branching edges, the terminal complement
at the other branching edge is \(2-f^q(x)\).  The positivity domain for the propagated incidence variables is
\[
        0<x<\vartheta_q:=\frac{q+2}{q+1}.
\]
Equivalently,
\[
        0<f^j(x)<2\qquad(0\leq j\leq q),
        \qquad 2-f^q(x)>0.
\]
\item Suppose that the initial value \(x_N>0\) making the initial branching edge
normal lies outside the corresponding propagation domain.  For an edge-to-vertex path with
\(q\geq1\), choose \(0<z_0<\eta_q\) with \(z_0<x_N\), for a path between
two branching edges, choose \(0<z_0<\vartheta_q\) with \(z_0<x_N\).
Replacing \(x_N\) by \(z_0\) makes the initial-edge product strict in the
supernormal direction and
preserves all internal positivity conditions.  Moreover,
\[
 f^q(z_0)\longrightarrow+\infty
 \quad(z_0\uparrow\eta_q),\qquad
 2-f^q(z_0)\longrightarrow0
 \quad(z_0\uparrow\vartheta_q),
\]
in the edge-to-vertex and edge-to-edge cases, respectively.
\end{enumerate}
\end{lem}

\begin{proof}
Use
\[
        f^m(x)=\frac{m-(m-1)x}{m+1-mx}.
\]
For \(q=0\) in (1) there is no intermediate degree-two vertex and the only
positivity condition is \(x=f^0(x)>0\).  Now suppose \(q\geq1\).  For the
edge-to-vertex internal path, the intermediate degree-two vertices are
those with
\(j<q\).  For \(0\leq j<q\), positivity of \(f^j(x)\) and of its complement
\(2-f^j(x)\) is controlled by the numerator \(j-(j-1)x\), the denominator
\(j+1-jx\), and
\[
        2-f^j(x)=\frac{j+2-(j+1)x}{j+1-jx}.
\]
The first boundary among these quantities, as \(j\) ranges from \(0\) to
\(q-1\), is \(x=(q+1)/q=1+1/q\).  At the same boundary the denominator of
\(f^q(x)\) vanishes from the positive side, while the terminal value itself
is used only at the branching vertex.  This proves (1).  For two branching
edges, the additional terminal condition is
\[
        2-f^q(x)=\frac{q+2-(q+1)x}{q+1-qx}>0,
\]
which gives \(x<(q+2)/(q+1)\).  This boundary is no larger than all earlier
positivity boundaries, and within this interval the preceding formula gives
\(0<f^j(x)<2\) for \(0\leq j\leq q\).  This proves (2).

For (3), replacing \(x_N\) by \(z_0<x_N\) multiplies the
product at the initial edge by \(z_0/x_N<1\); hence the initial-edge product
becomes strict in the supernormal direction.  Parts (1) and (2) give positivity of all
internal incidences and degree-two complements.  Finally, the denominator of
\(f^q(z_0)\) tends to zero from above as
\(z_0\uparrow\eta_q\), whereas the numerator of
\(2-f^q(z_0)\) tends to zero from above as
\(z_0\uparrow\vartheta_q\).  These are the asserted limits.
\end{proof}

\begin{lem}\label{lem:endpoint-exclusion}
Let \(H\) be a connected \(k\)-uniform hypertree with
\(\rho_0(H)<4^{1/k}\).  At \(t=1/2\), equivalently \(\alpha=0\), no positive
assignment of the variables in~\eqref{eq:normalized-incidence} can induce either a normal labeling or a consistently
strictly supernormal labeling with $\beta=1/4$.
Moreover, for any such assignment, if all conditions except
one residual in the convention of Lemma~\ref{lem:residual-sign} are normal,
then the remaining residual is strictly positive.
\end{lem}

\begin{proof}
At \(\alpha=0\), the adjacency spectral radius of loose cycle is \(4^{1/k}\), so
\(\beta=1/4\).  A normal labeling induced by such an assignment would give
\(\rho_0(H)=4^{1/k}\), and a consistently strictly supernormal labeling
would give \(\rho_0(H)>4^{1/k}\), by
Lemma~\ref{lem:label-comparison}.  Both contradict the hypothesis.
For the final assertion, a zero remaining residual would give a normal
labeling, while a negative remaining residual would give a consistently
strictly supernormal labeling by Lemma~\ref{lem:residual-sign}, since \(H\)
is a hypertree.  Hence the remaining residual must be positive.
\end{proof}

Fix \(k\geq3\).  In the formulas below, the displayed pendent-path lengths and path
parameters are held fixed and are suppressed from the function notation.

For \(E_{a,b,c}^{k}\), define
\begin{equation}\label{eq:HE}
 H_E(t)=4-\frac1t-f^{a-1}(t)-f^{b-1}(t)-f^{c-1}(t).
\end{equation}
For \(BD_n^{k}\), which has one branching vertex and one branching edge,
put \(q=n-5\geq0\) and define
\begin{equation}\label{eq:c-BD}
 c(t)=\frac{1}{t(2-t)^2}.
\end{equation}
Its branching-vertex residual is
\begin{equation}\label{eq:HBD}
 H_{BD,q}(t)=4-\frac1t-2t-f^q(c(t)).
\end{equation}
Its propagation domain is \(J_{BD}=(0,1/2]\) when \(q=0\).  When
\(q\geq1\), it is the interval on which
\begin{equation*}
        c(t)<1+\frac1q.
\end{equation*}

For a hypergraph formed from one branching edge by attaching \(s\geq3\)
pendent loose paths of lengths \(\ell_1,\ldots,\ell_s\) at distinct vertices
of that edge, define
\begin{equation}\label{eq:Phi}
        \Phi(t)=t^{s-2}\prod_{i=1}^s u_{\ell_i}(t).
\end{equation}
For \(G_{a,b:q:c,d}^{k}\), define
\begin{equation}\label{eq:LR-two-branch}
 L_G(t)=t u_a(t)u_b(t),
 \qquad
 R_G(t)=t u_c(t)u_d(t),
 \qquad z(t)=\frac1{R_G(t)}.
\end{equation}
The value propagated to the other branching edge is
\begin{equation}\label{eq:w-two-branch}
 w(t)=2-f^q(z(t)).
\end{equation}
Equivalently,
\begin{equation*}
 2-f^q(z)=\frac{q+2-(q+1)z}{q+1-qz}.
\end{equation*}
By Lemma~\ref{lem:propagation-domain}(2), its propagation domain is
\begin{equation*}
 J_G=\{t\in(0,1/2]:z(t)<\vartheta_q\},
 \qquad \vartheta_q=\frac{q+2}{q+1}.
\end{equation*}
On \(J_G\), define
\begin{equation}\label{eq:Psi}
        \Psi(t)=L_G(t)\bigl(2-f^q(z(t))\bigr).
\end{equation}

\begin{pop}\label{prop:branching-residuals}
Fix \(k\geq3\). For
\(0<t\leq1/2\), write
\[
        \alpha_t=\psi_k(t),\qquad
        \beta_t=\lambda_k(\alpha_t)^{-k}.
\]
The functions above have the following properties on their stated
\(t\)-domains.
\begin{enumerate}[(1)]
\item The function \(H_E\) is strictly increasing and tends to
\(-\infty\) as \(t\downarrow0\).  Hence \(H_E(1/2)\leq0\) implies
\(H_E(t)<0\) for every \(0<t<1/2\).
\item For every \(s\geq3\), the function \(\Phi\) is strictly increasing
and tends to \(0\) as \(t\downarrow0\).  Hence \(\Phi(1/2)\leq1\) implies
\(\Phi(t)<1\) for every \(0<t<1/2\).
\item The functions \(L_G\) and \(R_G\) are strictly increasing and
\(z=1/R_G\)
is strictly decreasing.  If \(J_G\) is nonempty, then
\(J_G=(\tau,1/2]\) for a unique \(\tau\in(0,1/2)\), and \(\Psi\) is
strictly increasing on \(J_G\), with \(\Psi(t)\to0\) as
\(t\downarrow\tau\).  For \(t\notin J_G\), the modified propagation assignment in
Lemma~\ref{lem:propagation-domain}(3) gives a consistently strictly
\((\alpha_t,\beta_t)\)-supernormal labeling.  Consequently,
if \(z(1/2)<\vartheta_q\) and \(\Psi(1/2)\leq1\), then
\(G_{a,b:q:c,d}^{k}\) admits the corresponding
\((\alpha_t,\beta_t)\)-labeling for every \(0<t<1/2\).
\item The domain \(J_{BD}\) is \((0,1/2]\) for \(q=0\), and is
\((\tau,1/2]\) for a unique \(\tau\in(0,1/2)\) when \(q\geq1\).
On this domain, \(H_{BD,q}\) is strictly increasing,
\begin{equation}\label{eq:HBD-half}
        H_{BD,q}(1/2)=\frac1{q+9}>0,
\end{equation}
and \(H_{BD,q}(t)\to-\infty\) as \(t\) approaches the left endpoint of
\(J_{BD}\) from the right. For
\(t\notin J_{BD}\), which can occur only when \(q\geq1\), the modified
propagation assignment in Lemma~\ref{lem:propagation-domain}(3) gives a consistently
strictly \((\alpha_t,\beta_t)\)-supernormal labeling.
\end{enumerate}
\end{pop}

\begin{proof}
For (1),~\eqref{eq:f-derivative} gives
\[
 H_E'(t)=\frac1{t^2}-
 \sum_{\ell\in\{a,b,c\}}(f^{\ell-1})'(t)\geq4-3=1>0.
\]
The term \(-1/t\) gives the asserted limit.

For (2),~\eqref{eq:uell-log} gives
\[
 \frac{d}{dt}\log\Phi(t)
 \geq \frac{s-2}{t}-\frac{2s}{3}>0
 \qquad(0<t<1/2).
\]
Indeed, for \(s=3\) the right-hand side is \(1/t-2>0\); for
\(s\geq4\) it is at least \(2(s-2)-2s/3>0\).  The factor \(t^{s-2}\)
gives \(\Phi(t)\to0\).

For (3), the same logarithmic derivative estimate shows that \(L_G\) and
\(R_G\) are strictly increasing.  Hence \(z=1/R_G\) is strictly decreasing and
tends to infinity as \(t\downarrow0\).  If \(J_G\) is nonempty, it therefore
has the stated form.  On \(J_G\), \(f^q\) is strictly increasing, so
\(2-f^q(z(t))\), and hence \(\Psi(t)\), is strictly increasing.
Formula~\eqref{eq:w-two-branch} gives the limit at \(\tau\).

If \(t\notin J_G\), choose \(0<z_0<\vartheta_q\), with \(z_0<z(t)\) and
sufficiently close to \(\vartheta_q\).  Then
\[
 R_G(t)z_0<1,
 \qquad
 L_G(t)\bigl(2-f^q(z_0)\bigr)<1.
\]
\makebox[\linewidth][l]{The two branching-edge products are therefore strict in the supernormal direction.}
Lemma~\ref{lem:propagation-domain}(3) supplies positive internal
incidences satisfying the normality equations; the hypergraph is a hypertree, so the
labeling is consistent.  The final assertion follows from strict
monotonicity on \(J_G\) and this modified propagation assignment outside \(J_G\).

For (4), the function \(c(t)=1/[t(2-t)^2]\) is strictly decreasing, tends
to infinity as \(t\downarrow0\), and satisfies \(c(1/2)=8/9\).  This proves
the assertion about \(J_{BD}\).  On that domain, \(c'(t)<0\) and
\((f^q)'(x)>0\), whence
\[
        H_{BD,q}'(t)>\frac1{t^2}-2\geq2.
\]
Moreover,
\[
 f^q(8/9)=\frac{q+8}{q+9},
\]
which proves~\eqref{eq:HBD-half}.  When \(q=0\), the term \(-1/t-c(t)\)
gives the left-boundary limit.  When \(q\geq1\), the denominator of
\(f^q(c(t))\) tends to zero from above at \(\tau\), giving the same limit.

Finally, if \(t\notin J_{BD}\), choose \(0<z_0<1+1/q\), with
\(z_0<c(t)\) and sufficiently close to \(1+1/q\).  The branching-edge
product is multiplied by \(z_0/c(t)<1\), while \(f^q(z_0)\) can be made
arbitrarily large.  Thus both the initial-edge product and the terminal
branching-vertex inequality are strict in the supernormal direction.  Positivity and consistency
follow as in part (3).
\end{proof}

Table~\ref{tab:equality-point-equations} lists the correspondence between the
hypergraphs in Proposition~\ref{prop:luman-inputs}(2) and the determining
equations defined above. Proposition~\ref{prop:branching-residuals} supplies
their monotonicity and boundary behavior.

\begin{table}[!htbp]
\centering
\footnotesize
\caption{Equations determining \(\alpha_1\) for the hypergraphs in
Proposition~\ref{prop:luman-inputs}(2)}
\label{tab:equality-point-equations}
\begin{tabularx}{\textwidth}{@{}>{\raggedright\arraybackslash}p{0.27\textwidth}
>{\raggedright\arraybackslash}p{0.33\textwidth}X@{}}
\toprule
hypergraph & hypergraph family and parameters & determining equation \\
\midrule
\(D_n=E_{1,1,n-2}\),
\(E_6=E_{1,2,2}\),
\(E_7=E_{1,2,3}\),
\(E_8=E_{1,2,4}\) &
\(E_{a,b,c}\) with the listed pendent-path lengths &
\(H_E(t)=0\), defined in~\eqref{eq:HE} \\
\addlinespace[1mm]
\(D'_n=F_{1,1,n-3}\),
\(B_n=F_{1,2,n-4}\), and the remaining listed \(F\)-family hypergraphs &
\(F_{\ell_1,\ell_2,\ell_3}\), so \(s=3\) &
\(\Phi(t)=1\), defined in~\eqref{eq:Phi} \\
\addlinespace[1mm]
\(H_{1,1,1,\ell}\), \(1\leq\ell\leq4\) &
\(H_{\ell_1,\ell_2,\ell_3,\ell_4}\), so \(s=4\) &
\(\Phi(t)=1\), defined in~\eqref{eq:Phi} \\
\addlinespace[1mm]
\(B'_n=G_{1,1:n-6:1,1}\),
\(\bar B_n=G_{1,1:n-7:1,2}\),
\(G_{1,1:q:1,3}\) \((0\leq q\leq5)\) &
\(G_{a,b:q:c,d}\) with the listed pendent-path lengths and internal path parameter &
\(\Psi(t)=1\), defined in~\eqref{eq:Psi} \\
\addlinespace[1mm]
\(BD_n\ (n\geq5)\) &
\(BD_n\), with one branching vertex and one branching edge; internal path parameter
\(q=n-5\) &
\(H_{BD,q}(t)=0\), defined in~\eqref{eq:HBD} \\
\bottomrule
\end{tabularx}
\end{table}

\begin{thm}\label{thm:unique-equality-point}
Let \(H\) be a hypergraph in the list in
Proposition~\ref{prop:luman-inputs}(2), other than a loose path. Then there is
a unique \(\alpha_1\in(0,1)\) such that
\[
\begin{aligned}
 &\rho_\alpha(H)-\lambda_k(\alpha)<0 &&(0\leq\alpha<\alpha_1),\\
 &\rho_{\alpha_1}(H)-\lambda_k(\alpha_1)=0,&\\
 &\rho_\alpha(H)-\lambda_k(\alpha)>0 &&(\alpha_1<\alpha<1).
\end{aligned}
\]
Here \(\alpha_1=\psi_k(t_1)\), where \(t_1\) is the unique solution, in the
corresponding propagation domain, of the equation assigned to \(H\) in
Table~\ref{tab:equality-point-equations}.
\end{thm}

\begin{proof}
We use the parameter \(t\) from~\eqref{eq:psi}, write the incidence weights
in the form~\eqref{eq:normalized-incidence}, and impose the normality
equations at every edge and degree-two vertex not mentioned below.

Let \(H=E_{a,b,c}^{k}\).  All conditions except the branching-vertex
residual \(H_E\) are normal.  Lemma~\ref{lem:endpoint-exclusion} gives
\(H_E(1/2)>0\).  Proposition~\ref{prop:branching-residuals}(1) then gives a
unique \(t_1\in(0,1/2)\) with \(H_E(t_1)=0\).  By
Lemma~\ref{lem:residual-sign}, the labeling is strictly subnormal for
\(t>t_1\), normal at \(t=t_1\), and consistently strictly supernormal for
\(t<t_1\).

If \(H\) belongs to one of the listed \(F\)- or \(H\)-families, suppose that its branching edge
contains \(s\geq3\) non-leaf vertices and carries
pendent paths of lengths \(\ell_1,\ldots,\ell_s\).  The only residual is
\(\Phi(t)-1\).  Lemma~\ref{lem:endpoint-exclusion} gives
\(\Phi(1/2)>1\), and Proposition~\ref{prop:branching-residuals}(2) gives a
unique \(t_1\in(0,1/2)\) with \(\Phi(t_1)=1\).  Hence
Lemma~\ref{lem:residual-sign} gives a strictly subnormal labeling for
\(t>t_1\), a normal labeling at \(t=t_1\), and a consistently strictly
supernormal labeling for \(t<t_1\).

Let \(H=G_{a,b:q:c,d}^{k}\), and make the right branching edge normal by
using the initial value \(z=1/R_G\) from~\eqref{eq:LR-two-branch}.  If
\(1/2\notin J_G\), Proposition~\ref{prop:branching-residuals}(3) gives a
consistently strictly supernormal labeling at \(\alpha=0\), contradicting
Lemma~\ref{lem:endpoint-exclusion}.  Hence \(J_G=(\tau,1/2]\).

On \(J_G\), the only remaining residual is \(\Psi(t)-1\).
Lemma~\ref{lem:endpoint-exclusion} gives \(\Psi(1/2)>1\), while
Proposition~\ref{prop:branching-residuals}(3) shows that \(\Psi\) increases
strictly from \(0\).  Thus it has a unique zero \(t_1\) of \(\Psi(t)-1\) in
\(J_G\).  Lemma~\ref{lem:residual-sign} gives a strictly subnormal labeling
for \(t>t_1\), a normal labeling at \(t=t_1\), and a consistently strictly
supernormal labeling for \(t<t_1\) in \(J_G\).  Outside \(J_G\) one has
\(t<t_1\), and Proposition~\ref{prop:branching-residuals}(3) supplies the
same consistently strictly supernormal conclusion.  Hence there is no
additional equality parameter.

Let \(H=BD_n,\ q=n-5\geq0\) be the internal path parameter of \(BD_n\).  The two length-one
pendent paths have value \(u_1(t)=2-t\) at the branching edge, and normality at that
edge gives the initial value \(c(t)\) in~\eqref{eq:c-BD}.  By
Proposition~\ref{prop:branching-residuals}(4), \(H_{BD,q}\) has a unique
zero \(t_1\) in \(J_{BD}\).  Lemma~\ref{lem:residual-sign} gives a strictly
subnormal labeling for \(t>t_1\), a normal labeling at \(t=t_1\), and a
consistently strictly supernormal labeling for \(t<t_1\) within \(J_{BD}\).
If \(t\notin J_{BD}\), then \(t<t_1\), and the modified propagation assignment in
Proposition~\ref{prop:branching-residuals}(4) gives the same consistently
strictly supernormal conclusion.  Thus no additional equality parameter occurs outside the
propagation domain.

In each non-path case we have obtained a unique \(t_1\in(0,1/2)\).  Define
\[
        \alpha_1=\psi_k(t_1)
        =\frac{1-2t_1}{1-2t_1+t_1^{2/k}}.
\]
At \(t=t_1\), the corresponding assignment is defined on its propagation
domain; all residuals vanish, and every remaining vertex and edge condition
is normal.  Since the hypergraphs are all hypertrees,
the resulting normal labeling is consistent, hence
Lemma~\ref{lem:label-comparison}(3) gives
\(\rho_{\alpha_1}(H)=\lambda_k(\alpha_1)\).  Since \(\psi_k\) is strictly
decreasing, \(t>t_1\) corresponds to \(\alpha<\alpha_1\), and \(t<t_1\)
corresponds to \(\alpha>\alpha_1\).  The strictly subnormal and consistently
strictly supernormal labelings obtained above, together with
Lemma~\ref{lem:label-comparison}, complete the proof.
\end{proof}

\begin{cor}\label{cor:equality-point-uniformity}
Let \(L^{(r)}\) be a hypergraph in the list in
Proposition~\ref{prop:luman-inputs}(2) with fixed
pendent-path lengths and internal path parameters, where \(r=3\) for a member
of the three-uniform list and \(r=4\) when
\(L^{(4)}=H_{1,1,1,\ell}^{(4)}\) for some \(1\leq\ell\leq4\).  For \(k\geq r\), let
\(L^{k}\) be its \(k\)-uniform extension, and let \(t_1(L)\) be
the unique solution of the determining equation for \(L^{(r)}\).  Every \(L^{k}\) has the
same solution \(t_1(L)\), and the value at which equality holds is
\[
        \alpha_1(L,k)=
        \frac{1-2t_1(L)}{1-2t_1(L)+t_1(L)^{2/k}}.
\]
Moreover, \(\alpha_1(L,k)\) is strictly decreasing in \(k\), and
\[
        \lim_{k\to\infty}\alpha_1(L,k)
        =\frac{1-2t_1(L)}{2-2t_1(L)}.
\]
\end{cor}

\begin{proof}
All four determining equations depend only on \(t\) and the fixed pendent-path lengths and
internal path parameters.  Uniform extension does not alter them, by the
computation following~\eqref{eq:edge-normal-a}.  Since \(0<t_1(L)<1\), the
function \(a\mapsto t_1(L)^a\) is strictly decreasing in the exponent \(a\).
As \(2/k\) decreases with \(k\), the quantity \(t_1(L)^{2/k}\) strictly
increases with \(k\), which proves the asserted
monotonicity and the limit.
\end{proof}

\section{Hypergraphs with adjacency spectral radius greater than that of a loose cycle}

In this section, we prove the following theorem.
\begin{theorem}
\label{thm:above-threshold}
Let \(k\geq3\), \(H\) be a finite connected \(k\)-uniform hypergraph. If
\(\rho_0(H)>4^{1/k}\), then
\[
        \rho_\alpha(H)-\lambda_k(\alpha)>0
        \qquad(0<\alpha<1).
\]
\end{theorem}

For an edge \(e\in E(H)\), we use \(H-e\) to denote the hypergraph obtained
by deleting \(e\) together with all leaf vertices of \(H\) contained in
\(e\); equivalently,
\[
 E(H-e)=E(H)\setminus\{e\}
 \quad\text{and}\quad
 V(H-e)=V(H)\setminus\{x\in e:d_H(x)=1\}.
\]
We say that \(e\) is a \(2\)-bridge of \(H\) if it contains exactly two
non-leaf vertices and \(H-e\) is disconnected. Let \(u\) and \(v\) be the
two non-leaf vertices of the \(2\)-bridge \(e\). The contraction, denoted
by \(H/e\), is the hypergraph obtained from \(H-e\) by identifying \(u\)
and \(v\) as a new vertex \(w\). In this case,
we also say \(H\) is an \emph{expansion} of \(H/e\) at \(w\). A hypergraph
\(H'\) has an expansion at \(w\) if and only if \(w\) is a cut vertex of
\(H'\), i.e.,\(H'=J_1\cup J_2\), \(J_1\cap J_2=\{w\}.\)

\begin{figure}[htbp]
    \centering
    \begin{tikzpicture}[scale=1.0]


        \draw (0,0) circle (1.15);
        \draw (3.1,0) circle (1.15);

        \node at (0,0) {\(J_1\)};
        \node at (3.1,0) {\(J_2\)};


        \node[left=2pt] at (1.25,-0.2) {\(u\)};
        \node[right=2pt] at (1.85,-0.2) {\(v\)};

        \hyperpath{1.95,0}{1.15,0}{1}



        \node at (1.55,0.3) {\(e\)};

        \node at (1.55,-1.55) {\(H\)};


        \begin{scope}[xshift=7cm]

            \draw (0,0) circle (1.15);
            \draw (2.3,0) circle (1.15);

            \node at (0,0) {\(J_1\)};
            \node at (2.3,0) {\(J_2\)};

            \filldraw[fill=red!] (1.15,0) circle (1.0pt);
            \node[left=3pt] at (1.15,0) {\(w\)};

            \node at (1.15,-1.55) {\(H/e\)};

        \end{scope}

    \end{tikzpicture}
\end{figure}

We have the following lemma.

\begin{lem}\label{lem:two-bridge-alpha}
Let $k\geq3$, $0\leq\alpha<1$, and let $H$ be a connected $k$-uniform
hypergraph with a $2$-bridge $e$.  Put $J=H/e$ and
$\lambda=\lambda_k(\alpha)$.  Then
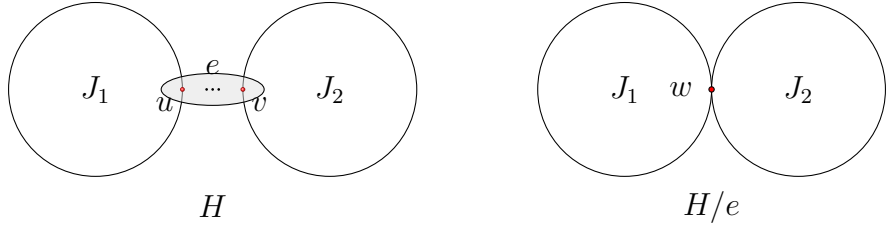
\begin{enumerate}[(1)]
\item if $\rho_\alpha(J)>\lambda$, then $\rho_\alpha(H)>\lambda$;
\item if $\rho_\alpha(J)=\lambda$, then $\rho_\alpha(H)\geq\lambda$.
Moreover, equality holds if and only if, for every consistent
$(\alpha,\lambda^{-k})$-normal weighted incidence matrix $B$ of $J$, the
weight at the contracted vertex $w$ splits evenly between the two
components of $H-e$.  Namely, if $\mathcal E_1$ and $\mathcal E_2$ are the
corresponding sets of edges incident with $w$, then
\[
 \sum_{f\in\mathcal E_1}B(w,f)
 =\frac12
 =\sum_{f\in\mathcal E_2}B(w,f).
\]
\end{enumerate}
\end{lem}

\begin{proof}
Put \(y=\alpha/\lambda\) and \(d=1/2-y>0\), by~\eqref{2alpha}. If the two non-leaf
incidences of \(e\) have weights \(p_1,p_2\), and its leaf incidences have
weight \(1\), then the incidence weights on the bridge edge satisfy the
supernormal edge inequality exactly when
\begin{equation}\label{eq:two-bridge-product}
        (p_1-y)(p_2-y)\leq d^2.
\end{equation}

First let \(\mu=\rho_\alpha(J)>\lambda\), and let \(B_0\) be the
consistent \((\alpha,\mu^{-k})\)-normal weighted incidence matrix associated with a
Perron vector of \(J\). Put \(c=\mu/\lambda\) and \(C=cB_0\). Since
\[
 C(v,f)-y=c\bigl(B_0(v,f)-\alpha/\mu\bigr),
\]
the matrix \(C\) is consistent, every old vertex sum is \(c\), and every
old edge is normal with parameter \(\lambda^{-k}\). Let \(X_i\) be the sum
over \(\mathcal E_i\) at \(w\), and put
\[
        a_i=\max\{1-X_i-y,0\}\qquad(i=1,2).
\]
Here \(X_1+X_2=c\). If \(a_1a_2>0\), then
\(a_1+a_2=2-c-2y<2d\), and otherwise \(a_1a_2=0\). Thus
\(a_1a_2<d^2\). For sufficiently small \(\varepsilon>0\), the choices
\[
        p_i=y+a_i+\varepsilon\qquad(i=1,2)
\]
make both new vertex sums greater than \(1\) and
\eqref{eq:two-bridge-product} strict. Since \(e\) lies on no Berge cycle,
the resulting labeling is consistent. Lemma~\ref{lem:label-comparison}
gives \(\rho_\alpha(H)>\lambda\).

Now suppose that \(\rho_\alpha(J)=\lambda\), and let \(B\) be any
consistent \((\alpha,\lambda^{-k})\)-normal weighted incidence matrix of \(J\).
Define \(X_i\) and \(a_i\) as above; now \(X_1+X_2=1\). If
\(X_1\ne X_2\), then \(a_1a_2<d^2\): this is clear when one \(a_i\)
vanishes, while otherwise \(a_1+a_2=2d\) and equality in the
arithmetic--geometric mean inequality would force \(X_1=X_2\). The same
small-\(\varepsilon\) construction therefore gives a consistently strictly
supernormal labeling of \(H\). If \(X_1=X_2=1/2\), taking
\(p_1=p_2=1/2\) gives a consistent normal labeling. Since \(J\) has a
consistent normal weighted incidence matrix obtained from a Perron vector,
Lemma~\ref{lem:label-comparison} proves (2),
including the stated equality criterion.
\end{proof}

Recall that $$\psi_k(t)=\frac{1-2t}{1-2t+t^{2/k}}.$$ For uniform extensions, we have the following lemma.

\begin{lem}\label{lem:uniform-extension-alpha}
Let $3\leq r\leq k$, let $H'$ be a connected $r$-uniform hypergraph, and
let $H$ be its $k$-uniform extension.  Fix $0<t\leq1/2$, and, for
$m\in\{r,k\}$, put
\[
 \alpha_m=\psi_m(t),\qquad
 \lambda_m=\lambda_m(\alpha_m),\qquad
 \beta_m=\lambda_m^{-m}.
\]
$H'$ admits a consistent $(\alpha_r,\beta_r)$-normal labeling if and only
if $H$ admits a consistent $(\alpha_k,\beta_k)$-normal labeling.  The same
equivalence holds for subnormal and supernormal labelings.  Strictness in
the latter two cases, and consistency whenever imposed, are preserved.
\end{lem}

\begin{proof}
For $m\in\{r,k\}$,
\[
 \frac{\alpha_m}{\lambda_m}
 =y:=\frac{1-2t}{2(1-t)},
\]
which is independent of $m$.  In a normal labeling every leaf incidence
has weight $1$.  In a subnormal or supernormal labeling we may also assume that every leaf
incidence has weight $1$ for replacing a leaf weight by $1$ preserves the
corresponding inequalities. If this removes the only strict vertex
inequality, then the incident edge inequality becomes strict.

Keep the same weights on the incidences inherited from $H'$ and assign
weight $1$ to every new leaf vertex.  For each non-leaf
incidence write
\[
        B(v,e)=y+t(1-y)a(v,e).
\]
If an $m$-edge has $s$ non-leaf incidences with variables
$a_1,\ldots,a_s$, then \eqref{eq:cycle} gives
\[
 \prod_{v\in e}
 \frac{B(v,e)-\alpha_m/\lambda_m}{1-\alpha_m}
 =\beta_m\,t^{s-2}\prod_{i=1}^s a_i.
\]
Hence the sign of every edge inequality relative to $\beta_m$ is independent
of $m$.  Uniform extension does not alter the vertex sums at the inherited
vertices, and every new leaf vertex has sum $1$.  Finally, every Berge cycle of $H$
corresponds to a Berge cycle of $H'$, and
$B(v,e)-\alpha_m/\lambda_m=t(1-y)a(v,e)$ on every non-leaf vertex.
Thus the consistency products are unchanged.  The same argument applies in
the reverse direction after deleting the added leaf vertices.
\end{proof}

\begin{cor}\label{cor:reduction-contraction-alpha}
Let \(H\) be a connected \(k\)-uniform hypergraph, and let \(G\) be a
connected \(r\)-uniform hypergraph, where \(3\leq r\leq k\), obtained from
\(H\) by a finite sequence of uniformity reductions (the inverses of uniform
extensions) and contractions of \(2\)-bridges. Fix \(0<t<1/2\), and put
\(\alpha_m=\psi_m(t)\) at every uniformity \(m\) occurring in the sequence.
If
\[
        \rho_{\alpha_r}(G)>\lambda_r(\alpha_r),
\]
then
\[
        \rho_{\alpha_k}(H)>\lambda_k(\alpha_k).
\]
\end{cor}

\begin{proof}
Let \(L\) be an \(m\)-uniform hypergraph occurring in the sequence, and
suppose that \(\mu=\rho_{\alpha_m}(L)>\lambda_m(\alpha_m)\). Scaling a
consistent normal weighted incidence matrix of \(L\), obtained from a Perron
vector, by \(\mu/\lambda_m(\alpha_m)\), as in the proof of
Lemma~\ref{lem:two-bridge-alpha}, gives a consistently strictly supernormal
labeling at the loose-cycle parameter. Lemma~\ref{lem:uniform-extension-alpha}
therefore preserves the strict comparison when a uniformity reduction is
reversed, while Lemma~\ref{lem:two-bridge-alpha}(1) preserves it when a
bridge contraction is reversed. Iteration proves the result.
\end{proof}

\begin{theorem}\label{thm:extension-expansion-alpha}
Let $3\leq r\leq k$, and let $G$ be a connected $r$-uniform hypergraph. Suppose that a $k$-uniform hypergraph $H$ can be
obtained from $G$ by uniform extension and a sequence of $2$-bridge
expansions. If
\[
\rho_\alpha(G)>\lambda_r(\alpha)
\qquad\text{for all } \alpha\in(0,1),
\]
then
\[
\rho_\alpha(H)>\lambda_k(\alpha)
\qquad\text{for all } \alpha\in(0,1).
\]
\end{theorem}

\begin{proof}
Fix $t\in(0,1/2)$ and put
\[
\alpha_r=\psi_r(t),\qquad \alpha_k=\psi_k(t).
\]
Since $\rho_\alpha(G)>\lambda_r(\alpha)$ for every $\alpha\in(0,1)$, we have
\[
\rho_{\alpha_r}(G)>\lambda_r(\alpha_r).
\]
Corollary~\ref{cor:reduction-contraction-alpha}, applied to the reverse
sequence, gives
\[
\rho_{\alpha_k}(H)>\lambda_k(\alpha_k).
\]
Since $\psi_k$ is strictly decreasing and bijectively maps \((0,1/2)\) onto \((0,1)\), it follows that
\[
\rho_\alpha(H)>\lambda_k(\alpha)
\]
for every $\alpha\in(0,1)$.
\end{proof}

Let $Se^{(k)}(d)$ denote the $k$-uniform hypergraph obtained from an edge
$e=(v_1,v_2,\ldots,v_k)$ by attaching a pendent edge to each vertex in $\{v_1,\ldots,v_d\}$. Shan et al. proved the following result.

\begin{lem}[{\cite[Definition~3.5 and Proposition~3.3]{ShanWangWang}}]\label{lem:r5}
Suppose $k\ge 3$, $\alpha\in[0,1)$. There exists $\alpha_0\in[0,1)$
such that for any $\alpha\in[\alpha_0,1)$, it is true that
\[
\rho_{\alpha}\bigl(Se^{(k)}(d)\bigr)>\lambda_k(\alpha).
\]
Furthermore, $\alpha_0>0$ for $d=3,4$ and $\alpha_0=0$ for $d\ge 5$.
\end{lem}

We now explain the strategy for proving Theorem~\ref{thm:above-threshold}.
Let $H$ be a connected $k$-uniform hypergraph with
$\rho_0(H)>4^{1/k}$.  By Proposition~\ref{subgraph}, it is enough to find
inside $H$ a connected subhypergraph for which the desired strict comparison
with $\lambda_k(\alpha)$ is already known.  Thus the purpose of the
structural analysis below is not to classify $H$ again, but to dispose of
every case in which such a subhypergraph is already present and to isolate
the few remaining cases requiring direct calculation.

Lemmas~\ref{sc:lem:cycle-subgraphs} and~\ref{sc:lem:large-intersection},
together with Proposition~\ref{subgraph}, give the required strict
comparison whenever $H$ is non-linear or contains a Berge cycle.  In the remaining
case, perform uniformity reductions while the current uniformity is at least
$4$ and the current hypergraph is reducible.  Let $G$ be the resulting
linear $r$-uniform hypertree.  Then either $r=3$, with no irreducibility
assumption on $G$, or $r\geq4$ and $G$ is irreducible.  Lemmas~\ref{lem:label-comparison}
and~\ref{lem:uniform-extension-alpha} show that
$\rho_0(G)>4^{1/r}$.  We first prove the strict comparison for $G$.
If $r\geq5$, irreducibility supplies an edge of $G$ without leaf vertices;
together with one adjacent edge at each of its vertices, it contains
$Se^{(r)}(r)$.
Lemma~\ref{lem:r5}, with \(d=r\geq5\), and Proposition~\ref{subgraph} prove
the desired comparison.  Hence only an irreducible $4$-uniform hypergraph and
an arbitrary $3$-uniform hypergraph remain.

If $G$ contains, as a proper subhypergraph, a hypergraph $G'$ listed in
Proposition~\ref{prop:luman-inputs}(1), then $G'$ is also a hypertree and
is therefore neither a loose cycle nor $C_2^r$.  Consequently,
Theorem~\ref{thm:equality-families} and Proposition~\ref{subgraph} give
\[
 \rho_\alpha(G)>\rho_\alpha(G')>\lambda_r(\alpha)
 \qquad(0<\alpha<1).
\]
On the other hand, a hypergraph listed in
Proposition~\ref{prop:luman-inputs}(2) cannot be the present $G$, because
its adjacency spectral radius is strictly less than $4^{1/r}$.  Consequently
we need only consider the cases excluded by neither observation.  The
following case analysis shows that every such $G$ contains a connected
subhypergraph $G'$ from which a possibly empty sequence of $2$-bridge
contractions produces either a member of the list in
Proposition~\ref{prop:luman-inputs}(1) or one of the seven hypergraphs \(
H_1,\ H_2,\ H_{1,1,1,5}^{(4)},\ G_1,\
F_{3,3,3}^{3},\ G_{1,1:0:2,2}^{3},\ M.
\)

The four auxiliary hypertrees not belonging to the families introduced
in Section~\ref{sec:structures} are defined below and illustrated in
Fig.~\ref{fig:core-hypergraphs}.

\begin{enumerate}[(1)]
	\item \(H_1\) is the \(4\)-uniform hypertree with a central edge
	\[
	e=\{v_1,v_2,v_3,v_4\}.
	\]
	Two pendent edges are attached at \(v_1\), and one pendent edge is
	attached at each of \(v_2,v_3,v_4\).
	
	\item \(H_2\) is the \(4\)-uniform hypertree constructed from a central
	edge
	\[
	e=\{v_1,v_2,v_3,v_4\}
	\]
	and a second edge
	\[
	f=\{v_1,u_1,u_2,w\},
	\qquad e\cap f=\{v_1\}.
	\]
	One pendent edge is attached at each of \(v_2,v_3,v_4,u_1,u_2\),
	while \(w\) is a leaf vertex.
	
	\item \(G_1\) is the \(3\)-uniform hypertree with a branching vertex
	\(v\) of degree three. A pendent loose path of length \(1\) and a
	pendent loose path of length \(2\) are attached at \(v\). The third
	edge incident with \(v\) is
	\[
	e=\{v,x,y\},
	\]
	and one pendent edge is attached at each of \(x\) and \(y\).
	
	\item \(M\) is the \(3\)-uniform hypertree containing three consecutive
	branching edges \(e_1,e_2,e_3\), where
	\[
	e_1=\{x,a_1,a_2\},\qquad
	e_2=\{x,y,z\},\qquad
	e_3=\{y,b_1,b_2\}.
	\]
	One pendent edge is attached at each of
	\[
	a_1,\ a_2,\ z,\ b_1,\ b_2.
	\]
\end{enumerate}

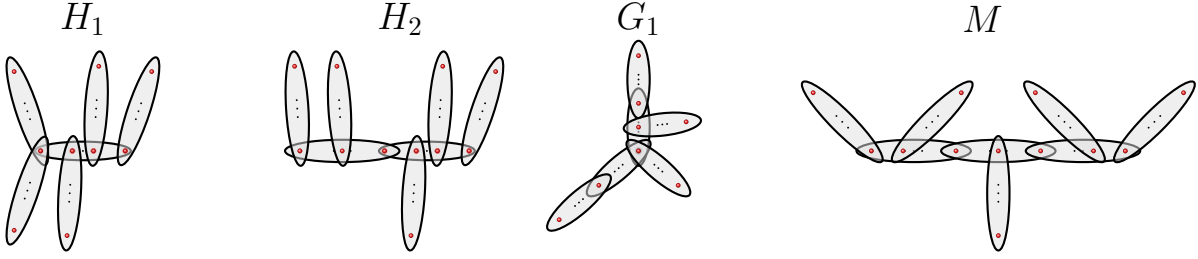
\begin{figure}[H]
\centering
	\begin{tikzpicture}[thick,scale=0.7]
		\begin{scope}[shift={(-6.5,0)}]
		\node at(0,2.5){\large$H_1$};
	      \coordinate (v1) at (-0.8,0);
		\coordinate (v2) at (-0.2,0);
        \coordinate (v3) at (0.2,0);
		\coordinate (v4) at (0.8,0);	
		\fitellipsis[gray]{v1}{v4}{3}{5}
		\node[circle,fill=red!80,inner sep=0.8pt] at (v1) {};
		\node[circle,fill=red!80,inner sep=0.8pt] at (v2) {};
		\node[circle,fill=red!80,inner sep=0.8pt] at (v3) {};
		\node[circle,fill=red!80,inner sep=0.8pt] at (v4) {};
		\coordinate (a) at (-1.3,1.5);
		\hyperpath{v1}{a}{1}
        \coordinate (e) at (-1.3,-1.5);
		\hyperpath{v1}{e}{1}
		\coordinate (b) at (-0.3,-1.6);
		\hyperpath{v2}{b}{1}
        \coordinate (c) at (0.3,1.6);
		\hyperpath{v3}{c}{1}
		\coordinate (d) at (1.3,1.5);
		\hyperpath{v4}{d}{1}
		\end{scope}
		\begin{scope}[shift={(0,0)}]
		\node at(-0.5,2.5){\large$H_2$};
	      \coordinate (v1) at (-0.8,0);
		\coordinate (v2) at (-0.2,0);
        \coordinate (v3) at (0.2,0);
		\coordinate (v4) at (0.8,0);	
		\fitellipsis[gray]{v1}{v4}{3}{5}
		\node[circle,fill=red!80,inner sep=0.8pt] at (v1) {};
		\node[circle,fill=red!80,inner sep=0.8pt] at (v2) {};
		\node[circle,fill=red!80,inner sep=0.8pt] at (v3) {};
		\node[circle,fill=red!80,inner sep=0.8pt] at (v4) {};
		\coordinate (a) at (-2.4,0);
		\hyperpath{v1}{a}{1}
        \coordinate (e) at (-2.5,1.6);
		\hyperpath{e}{a}{1}
        \coordinate (f) at (-1.6,0);
        \coordinate (g) at (-1.7,1.6);
		\hyperpath{f}{g}{1}
		\coordinate (b) at (-0.3,-1.6);
		\hyperpath{v2}{b}{1}
        \coordinate (c) at (0.3,1.6);
		\hyperpath{v3}{c}{1}
		\coordinate (d) at (1.3,1.5);
		\hyperpath{v4}{d}{1}
		\end{scope}
		\begin{scope}[shift={(4,0)}]
			\node at(0,2.5){\large$G_1$};
			\bver{0,0}
			\hyperpath{0,0}{0,1.8}{2}
			\hyperpath{0,0}{-1.5,-1.3}{2}
			\hyperpath{0,0}{0.75,-0.65}{1}
            \hyperpath{0,0.45}{0.9,0.55}{1}
		\end{scope}
		\begin{scope}[shift={(10,0)}]
        \node at(0.5,2.5){\large$M$};
        \coordinate (A) at (-1.6,0) ;
        \coordinate (B) at (2.4,0);
        \coordinate (C) at (-1,0);
        \coordinate (D) at (3.2,0);
        \hyperpath{A}{D}{3}  
        \hyperpath{A}{-2.7,1.1}{1}
        \hyperpath{-1,0}{0.1,1.1}{1}
        \hyperpath{D}{4.3,1.1}{1}
        \hyperpath{2.6,0}{1.5,1.1}{1}
        \hyperpath{0.8,-1.6}{0.8,0}{1}
		\end{scope}
	\end{tikzpicture}
\caption{The hypergraphs $H_1$, $H_2$, $G_1$, and $M$.}
\label{fig:core-hypergraphs}
\end{figure}

We treat the two remaining uniformities separately.  The resulting $4$-uniform
hypergraph is irreducible, whereas uniformity reduction stops at $3$ without
imposing irreducibility on the resulting $3$-uniform hypergraph.

\begin{lem}\label{lem:four-core-reduction}
	Let $H$ be an irreducible linear $4$-uniform hypertree satisfying
	$\rho_0(H)>4^{1/4}$. Then $H$ has a connected
	subhypergraph $H'$ from which a possibly empty sequence of
	$2$-bridge contractions produces either a member of the list in
	Proposition~\ref{prop:luman-inputs}(1) or one of the following three
	hypergraphs:
	\[
	H_1,\ H_2,\ H_{1,1,1,5}^{(4)}.
	\]
\end{lem}

\begin{proof}
    Since $H$ is irreducible, it has an edge
	$e$ none of whose four vertices is a leaf.  Write
	$e=\{v_1,v_2,v_3,v_4\}$, and let $T_i$ be the component of $H-e$
	containing $v_i$.  Each $T_i$ contains at least one edge.  Suppose first
	that every $T_i$ is a loose path and $v_i$ is an end vertex of $T_i$ for $i\in\{1,2,3,4\}$. In this case
	$H=H_{a,b,c,d}^{(4)}$, where $1\leq a\leq b\leq c\leq d$.  If
	$c\geq2$, then $H$ contains $H_{1,1,2,2}^{(4)}$, which is listed in
	Proposition~\ref{prop:luman-inputs}(1).  It is a proper subhypergraph,
	because $\rho_0(H)>4^{1/4}$ whereas
	$\rho_0(H_{1,1,2,2}^{(4)})=4^{1/4}$.  If $c=1$, then $a=b=c=1$.
	For $d\leq4$, Proposition~\ref{prop:luman-inputs}(2) gives
	$\rho_0(H)<4^{1/4}$, contrary to the hypothesis; for $d\geq5$, $H$
	contains $H_{1,1,1,5}^{(4)}$.
	
	Next suppose that some $T_i$ is a loose path but $v_i$ is an internal
	vertex of that path.  Retain $e$, the two edges of $T_i$ incident with
	$v_i$, and one edge from each $T_j$ with $j\ne i$. The resulting connected
	subhypergraph is $H_1$. Thus $H$ contains $H_1$ as a subhypergraph.
	
	It remains to consider the case in which some $T_i$ is not a loose path.
	Root $T_i$ at $v_i$, and let $z$ be the first branching vertex or branching edge encountered
	from the root. The part between $e$ and $z$ is a loose path.  If $z$ is a
	branching vertex, retain this path, two edges leaving $z$, and one edge from
	each of the other three components of $H-e$.  Every internal edge of the
	retained path is a $2$-bridge, and contracting these edges produces $H_1$.
	If $z$ is a branching edge, then, besides the non-leaf vertex of $z$ lying
	toward $e$, the edge $z$ has at least two other non-leaf vertices.  Choose
	two of them and retain one edge leaving $z$ at each, together with the path
	from $e$ to $z$ and one edge from each of the other three components of
	$H-e$.  Contracting the internal $2$-bridges of the retained path produces
	$H_2$.  These cases exhaust all possibilities for an irreducible linear $4$-uniform hypertree.
\end{proof}

We now turn to the $3$-uniform case.

\begin{lem}\label{lem:three-core-reduction}
Let $H$ be a linear $3$-uniform hypertree satisfying
$\rho_0(H)>4^{1/3}$. Then $H$ has a connected
subhypergraph $H'$ from which a possibly empty sequence of
$2$-bridge contractions produces either a member of the list in
Proposition~\ref{prop:luman-inputs}(1) or one of the following four
hypergraphs:
\[
\ G_1,\
 F_{3,3,3}^{3},\ G_{1,1:0:2,2}^{3},\ M.
\]
\end{lem}

\begin{proof}
Since \(\mathcal I(H)\) is a tree, the components of
\(\mathcal I(H)-v\), together with \(v\), determine the branches at a
vertex \(v\); the analogous convention applies to an edge-vertex of
\(\mathcal I(H)\).

If $H$ contains a vertex of degree at least four, then $H$ contains $S_4^3$.  If two
vertices have degree three, the path joining them, together with two edges
leaving each end vertex, contains some $\Dtil_n^3$.  The hypergraphs
$S_4^3$ and $\Dtil_n^3$ are listed in
Proposition~\ref{prop:luman-inputs}(1).  They must occur properly, since
their adjacency spectral radius is $4^{1/3}$ and
$\rho_0(H)>4^{1/3}$.  We may therefore assume that $H$ has at most one
vertex of degree three and no vertex of larger degree.

Suppose that $v$ is the unique vertex of degree three.  First assume that
none of its three branches contains a branching edge.  Then
$H=E_{a,b,c}^3$, where $a\leq b\leq c$.  If $a\geq2$, then $H$ contains
$\Etil_6^3=E_{2,2,2}^3$.  Thus we may take $a=1$.  If $b\geq3$, then
$H$ contains $\Etil_7^3=E_{1,3,3}^3$; if $b=2$ and $c\geq5$, it contains
$\Etil_8^3=E_{1,2,5}^3$.  These three hypergraphs are listed in
Proposition~\ref{prop:luman-inputs}(1).  The remaining possibilities are
$E_{1,1,c}^3=D_{c+2}^3$ and $E_{1,2,c}^3$ with $2\leq c\leq4$; by
Proposition~\ref{prop:luman-inputs}(2), each has adjacency spectral radius
strictly less than $4^{1/3}$ and hence cannot be $H$.

Next suppose that one of the branches at $v$ contains a branching edge.
At least one branch has length one, since otherwise the first two edges in
each branch form $\Etil_6^3$.  Fix such a length-one branch.  If another
branch has at least two edges, retain its first two edges.  In the remaining
branch, retain the path from $v$ to the first branching edge and one edge
leaving each of the other two non-leaf vertices of that branching edge.
After contracting the $2$-bridges in the retained path, the resulting
hypergraph is $G_1$.  Otherwise the first two branches at $v$ both have
length one.  Let $e$ be the first branching edge in the third branch.  If
one of the two branches leaving $e$ has at least two edges, the retained
subhypergraph contains $\BDtil_n^3$ for some $n\geq6$, which is listed in
Proposition~\ref{prop:luman-inputs}(1).  If both have length one, then
$H\cong BD_n^3$ for some $n\geq5$, and
Proposition~\ref{prop:luman-inputs}(2) gives
$\rho_0(H)<4^{1/3}$.  This completes the case in which $H$ has a vertex of
degree three.

We may now assume that every vertex has degree at most two and classify
$H$ by the number of branching edges.  If there is no branching edge, then
$H$ is a loose path, so Proposition~\ref{prop:luman-inputs}(2) contradicts
$\rho_0(H)>4^{1/3}$.

If there is exactly one branching edge, write
$H=F_{a,b,c}^3$, where $a\leq b\leq c$.  If $a\geq3$, then $H$ contains
$F_{3,3,3}^3$.  Suppose $a=2$.  When $b\geq3$ and $c\geq4$, $H$ contains
$F_{2,3,4}^3$; the only remaining possibility with $b\geq3$ is
$F_{2,3,3}^3$.  When $b=2$ and $c\geq7$, $H$ contains
$F_{2,2,7}^3$; for $c\leq6$, it is one of the hypergraphs
$F_{2,2,c}^3$.  Now suppose $a=1$.  If $b\geq5$ and $c\geq6$, $H$
contains $F_{1,5,6}^3$.  If $b=4$ and $c\geq8$, it contains
$F_{1,4,8}^3$; if $b=3$ and $c\geq14$, it contains
$F_{1,3,14}^3$.  The five hypergraphs just used as subhypergraphs are
listed in Proposition~\ref{prop:luman-inputs}(1).  All possibilities not
covered by them are
\[
 F_{2,3,3}^3,\quad F_{2,2,c}^3\ (2\leq c\leq6),\quad
 F_{1,5,5}^3,\quad F_{1,4,c}^3\ (4\leq c\leq7),
\]
\[
 F_{1,3,c}^3\ (3\leq c\leq13),\quad
 D_n'^3,\quad B_n^3,
\]
and Proposition~\ref{prop:luman-inputs}(2) says that each of these has
adjacency spectral radius strictly less than $4^{1/3}$.  Hence this case
reduces either to a proper member of the list in
Proposition~\ref{prop:luman-inputs}(1) or to $F_{3,3,3}^3$.

If there are exactly two branching edges, write
$H=G_{a,b:q:c,d}^3$, with $a\leq b$ and $c\leq d$.  If $a+b\geq3$ and
$c+d\geq3$, then $H$ contains
$G_{1,2:q:1,2}^3=\Btil_{q+8}^3$.  This subhypergraph is proper, since
$\rho_0(H)>4^{1/3}$ whereas $\rho_0(\Btil_{q+8}^3)=4^{1/3}$.  Otherwise,
after interchanging the two ends, we may
assume that $a=b=1$.  If $c\geq2$, retain
$H'=G_{1,1:q:2,2}^3$.  Contracting the internal $2$-bridges of $H'$ produces
$G_{1,1:0:2,2}^3$.

It remains to consider $c=1$.  If $d=1$ or $2$, then
$H=B_{q+6}'{}^3$ or $\bar B_{q+7}^3$, respectively, and
Proposition~\ref{prop:luman-inputs}(2) gives a contradiction.  If $d=3$,
the same proposition excludes $0\leq q\leq5$, while $q=6$ gives
$G_{1,1:6:1,3}^3$, which is listed in
Proposition~\ref{prop:luman-inputs}(1), again contrary to the hypothesis.
For $q\geq7$, successively
contracting $q-6$ internal $2$-bridges produces
$G_{1,1:6:1,3}^3$.  Finally, if $d\geq4$, retain
$H'=G_{1,1:q:1,4}^3$.  Contracting the internal $2$-bridges of $H'$ produces
$G_{1,1:0:1,4}^3$, which is also listed in
Proposition~\ref{prop:luman-inputs}(1).

Finally, suppose that $H$ has at least three branching edges.  In the minimal
subtree of $\mathcal I(H)$ containing the edge-vertices corresponding to the
branching edges, choose three that are consecutive.  Retain the paths joining
them; at each of the two end
branching edges retain one edge in each of the two directions away from the
middle branching edge, and at the middle branching edge retain one edge in
its third direction.  Every internal edge on the two joining paths is a
$2$-bridge of the retained subhypergraph.  Contracting these edges produces
$M$.  This proves the lemma.
\end{proof}

In either Lemma~\ref{lem:four-core-reduction} or
Lemma~\ref{lem:three-core-reduction}, the hypergraph obtained by
the indicated contractions is a hypertree.  If it belongs to the list in
Proposition~\ref{prop:luman-inputs}(1), it is therefore neither a loose
cycle nor $C_2^r$, and Theorem~\ref{thm:equality-families} gives the
required strict comparison.  Hence it remains only to treat the seven hypergraphs \(
H_1,\ H_2,\ H_{1,1,1,5}^{(4)},\ G_1,\
F_{3,3,3}^{3},\ G_{1,1:0:2,2}^{3},\ M.
\)

\begin{proof}[Proof of Theorem~\ref{thm:above-threshold}]
Apply the reductions above to $H$.  In the remaining case, they yield an
$r$-uniform hypergraph $G$, where $r\in\{3,4\}$, a connected subhypergraph
$G'\subseteq G$, and a hypergraph $K$ obtained from $G'$ by the indicated
$2$-bridge contractions.  It remains to treat
\[
 H_1,\ H_2,\ H_{1,1,1,5}^{(4)},\ G_1,\
 F_{3,3,3}^{3},\ G_{1,1:0:2,2}^{3},\ M.
\]
Put
\[
        \alpha_t=\psi_r(t),\qquad
        \beta_t=\lambda_r(\alpha_t)^{-r}
        \qquad(0<t\leq1/2).
\]
By the propagation rules above, a pendent path of length one has value
\(t\) at a branching vertex and \(2-t\) at a branching edge. With \(c(t)\)
as in~\eqref{eq:c-BD}, put
\begin{equation}\label{eq:four-core-residuals}
\begin{aligned}
 h(t)&=4-\frac1t-2t,
 &R_{H_1}(t)&=t^2(2-t)^3h(t)-1,\\
 R_{H_2}(t)&=t^2(2-t)^3(2-c(t))-1,
 &R_{G_1}(t)&=4-\frac1t-t-f(t)-c(t),\\
 R_M(t)&=t(2-t)(2-c(t))^2-1.
\end{aligned}
\end{equation}
Set
\[
 J_{H_1}=\{t\in(0,1/2]:h(t)>0\},\qquad
 J_{H_2}=J_M=\{t\in(0,1/2]:c(t)<2\},\qquad
 J_{G_1}=(0,1/2].
\]
All unspecified incidence values are determined by the normality equations. On
the stated domains, the following assignments are positive and have only
the indicated residual.
\begin{enumerate}[(1)]
\item For \(H_1\), set
\[
 a(v_1,e)=h(t),\quad a(v_i,e)=2-t\quad(2\leq i\leq4),\quad
 a(v_1,e')=t
\]
for the two pendent edges \(e'\) at \(v_1\); the residual is \(R_{H_1}\)
at \(e\).

\item For \(H_2\), set
\[
 a(v_1,f)=c(t),\quad a(v_1,e)=2-c(t),\quad
 a(u_j,f)=2-t,\quad a(v_i,e)=2-t
\]
for \(j=1,2\) and \(i=2,3,4\); the residual is \(R_{H_2}\) at \(e\).

\item For \(G_1\), the pendent paths of lengths one and two have values
\(t\) and \(f(t)\) at \(v\).  Set
\(a(v,e)=c(t)\) and \(a(x,e)=a(y,e)=2-t\); the residual is \(R_{G_1}\)
at \(v\).

\item For \(M\), set
\[
\begin{gathered}
 a(x,e_1)=a(y,e_3)=c(t),\qquad
 a(x,e_2)=a(y,e_2)=2-c(t),\\
 a(a_i,e_1)=a(b_i,e_3)=2-t\quad(i=1,2),\qquad a(z,e_2)=2-t.
\end{gathered}
\]
The residual is \(R_M\) at \(e_2\).
\end{enumerate}

Now \(h'(t)=t^{-2}-2>0\) and
\[
 \frac{d}{dt}\bigl(t^2(2-t)^3\bigr)
 =t(2-t)^2(4-5t)>0,
\]
so \(R_{H_1}\) is strictly increasing on \(J_{H_1}\). Also,
\[
 R_{G_1}'(t)=\frac1{t^2}-1-\frac1{(2-t)^2}-c'(t)>0.
\]
On \(J_{H_2}=J_M\), both factors in
\[
 R_{H_2}(t)+1=t(2-t)\bigl(2t(2-t)^2-1\bigr)
\]
are positive and strictly increasing, and the same is true of the factors in
\(R_M(t)+1=t(2-t)(2-c(t))^2\). Hence the four residuals are strictly
increasing on their stated domains. Moreover,
\begin{equation}\label{eq:four-core-endpoints}
 R_{H_1}(1/2)=-\frac5{32},\quad
 R_{H_2}(1/2)=-\frac1{16},\quad
 R_{G_1}(1/2)=-\frac1{18},\quad
 R_M(1/2)=-\frac2{27}.
\end{equation}
Thus they are negative throughout those domains.

On the complementary domains, modify the assignments as follows.
\begin{enumerate}[(i)]
\item If \(t\notin J_{H_1}\), replace \(a(v_1,e)=h(t)\) by \(1\).
The nonzero residuals at \(v_1\) and \(e\) are
\(h(t)-1\) and \(t^2(2-t)^3-1\), respectively.
\item If \(t\notin J_{H_2}\), set \(a(v_1,f)=a(v_1,e)=1\).
The nonzero residuals at \(f\) and \(e\) are
\(c(t)^{-1}-1\) and \(t^2(2-t)^3-1\), respectively.
\item If \(t\notin J_M\), set
\[
 a(x,e_1)=a(x,e_2)=a(y,e_2)=a(y,e_3)=1.
\]
The residual at each side edge is \(c(t)^{-1}-1\), and that at \(e_2\)
is \(t(2-t)-1\).
\end{enumerate}
All these residuals are negative, so the modified assignments are positive
and induce strictly supernormal labelings. Since \(H_1,H_2,G_1\), and \(M\)
are hypertrees, all these labelings are consistent.
Lemmas~\ref{lem:residual-sign}
and~\ref{lem:label-comparison} therefore give
\[
 \rho_\alpha(H_i)>\lambda_4(\alpha)\quad(i=1,2),\qquad
 \rho_\alpha(G_1)>\lambda_3(\alpha),\qquad
 \rho_\alpha(M)>\lambda_3(\alpha)
\]
for every \(0<\alpha<1\).

For \(H_{1,1,1,5}^{(4)}\) and \(F_{3,3,3}^{3}\), the residual is
\(\Phi-1\); for \(G_{1,1:0:2,2}^{3}\), it is \(\Psi-1\). Each of these
hypergraphs \(L\), with \(r\) its uniformity, satisfies
\(\rho_0(L)>4^{1/r}\)~\cite[Theorems~1--2 and~4--5]{LuMan}. Whenever the propagation is defined at
\(t=1/2\), the associated positive incidence assignment has only the
indicated residual. That residual is negative; otherwise
Lemmas~\ref{lem:residual-sign} and~\ref{lem:label-comparison} would imply
\(\rho_0(L)\leq4^{1/r}\).
Proposition~\ref{prop:branching-residuals}(2),(3) then shows that these
residuals remain negative throughout their propagation domains, while
part~(3) supplies a consistently strictly supernormal labeling outside
\(J_G\) for \(G_{1,1:0:2,2}^{3}\).
Lemmas~\ref{lem:residual-sign} and~\ref{lem:label-comparison}, together with
Theorem~\ref{thm:equality-families}, now give the strict comparison for every
possible \(K\).

Finally, Lemma~\ref{lem:two-bridge-alpha}(1) lifts the comparison from \(K\)
to \(G'\), Proposition~\ref{subgraph} from \(G'\) to \(G\), and
Theorem~\ref{thm:extension-expansion-alpha} from \(G\) to \(H\).
\end{proof}

Combining Theorems~\ref{neverexceedmain}, \ref{thm:equality-families},
\ref{thm:unique-equality-point}, and~\ref{thm:above-threshold} immediately
yields Theorem~\ref{thm:global-five-type}.

\section{Classification for a fixed parameter}\label{sec:fixed-alpha}
For \(k\geq3\) and \(0<\alpha<1\), set
\[
 t_{k,\alpha}=\frac{\lambda_k(\alpha)-2\alpha}
 {2(\lambda_k(\alpha)-\alpha)}.
\]

By Theorem~\ref{thm:unique-equality-point} and
Proposition~\ref{prop:branching-residuals}, the family
\(\mathcal B_{k,\alpha}\) from Definition~\ref{def:fixed-alpha-family}
has the following explicit description. It consists of the non-path hypergraphs in
Proposition~\ref{prop:luman-inputs}(2) which, with \(t=t_{k,\alpha}\),
satisfy the condition in the corresponding row:
\[
\begin{array}{c@{\qquad}l}
 E_{a,b,c}^{k} & H_E(t)\geq0,\\
 BD_n^{k},\ q=n-5 & t\in J_{BD}\ \text{and}\ H_{BD,q}(t)\geq0,\\
 F_{a,b,c}^{k}\ \text{or}\ H_{a,b,c,d}^{k} & \Phi(t)\geq1,\\
 G_{a,b:q:c,d}^{k} & t\in J_G\ \text{and}\ \Psi(t)\geq1.
\end{array}
\]

\begin{proof}[Proof of Theorem~\ref{thm:fixed-alpha-classification}]
By Theorem~\ref{thm:global-five-type}, loose paths lie strictly below
\(\lambda_k(\alpha)\), loose cycles and \(C_2^k\) attain it, and the third
and fifth alternatives lie strictly above it. In the fourth alternative,
Theorem~\ref{thm:unique-equality-point} and
Proposition~\ref{prop:branching-residuals} show that
\(\rho_\alpha(H)\leq\lambda_k(\alpha)\) precisely when
\(H\in\mathcal B_{k,\alpha}\). This proves both directions.
\end{proof}

\bibliographystyle{elsarticle-num}  
\bibliography{ref}  

@article{CooperDutle,
  author  = {Cooper, Joshua and Dutle, Aaron},
  title   = {Spectra of uniform hypergraphs},
  journal = {Linear Algebra and its Applications},
  volume  = {436},
  number  = {9},
  year    = {2012},
  pages   = {3268--3292},
  doi     = {10.1016/j.laa.2011.11.018}
}

@article{Qi2005,
  author  = {Qi, Liqun},
  title   = {Eigenvalues of a real supersymmetric tensor},
  journal = {Journal of Symbolic Computation},
  volume  = {40},
  number  = {6},
  year    = {2005},
  pages   = {1302--1324},
  doi     = {10.1016/j.jsc.2005.05.007}
}

@article{LuMan,
  author  = {Lu, Linyuan and Man, Shoudong},
  title   = {Connected hypergraphs with small spectral radius},
  journal = {Linear Algebra and its Applications},
  volume  = {509},
  year    = {2016},
  pages   = {206--227},
  doi     = {10.1016/j.laa.2016.07.013}
}

@article{ShanWangWang,
  author  = {Shan, Haiying and Wang, Zhiyi and Wang, Feifei},
  title   = {{$(\alpha,\beta)$}-labelling method for {$k$}-uniform hypergraph and its applications},
  journal = {Linear Algebra and its Applications},
  volume  = {642},
  year    = {2022},
  pages   = {30--49},
  doi     = {10.1016/j.laa.2022.02.005}
}

@incollection{Smith,
  author    = {Smith, J. H.},
  title     = {Some properties of the spectrum of a graph},
  editor    = {Guy, Richard K. and Hanani, Haim and Sauer, Norbert and Sch{\"o}nheim, J.},
  booktitle = {Combinatorial Structures and Their Applications},
  publisher = {Gordon and Breach},
  address   = {New York},
  year      = {1970},
  pages     = {403--406}
}

@article{Nikiforov,
  author  = {Nikiforov, Vladimir},
  title   = {Merging the {$A$}- and {$Q$}-spectral theories},
  journal = {Applicable Analysis and Discrete Mathematics},
  volume  = {11},
  number  = {1},
  year    = {2017},
  pages   = {81--107},
  doi     = {10.2298/AADM1701081N}
}

@article{WangShanWang2020,
  author  = {Wang, Feifei and Shan, Haiying and Wang, Zhiyi},
  title   = {On some properties of the {$\alpha$}-spectral radius of the {$k$}-uniform hypergraph},
  journal = {Linear Algebra and its Applications},
  volume  = {589},
  year    = {2020},
  pages   = {62--79},
  doi     = {10.1016/j.laa.2019.12.010}
}

@article{GuoZhou2020,
  author  = {Guo, Haiyan and Zhou, Bo},
  title   = {On the {$\alpha$}-spectral radius of uniform hypergraphs},
  journal = {Discussiones Mathematicae Graph Theory},
  volume  = {40},
  number  = {2},
  year    = {2020},
  pages   = {559--575},
  doi     = {10.7151/dmgt.2268}
}

@article{HouChangShi2020,
  author  = {Hou, Yuan and Chang, An and Shi, Chao},
  title   = {On the {$\alpha$}-spectra of uniform hypergraphs and its associated graphs},
  journal = {Acta Mathematica Sinica, English Series},
  volume  = {36},
  number  = {7},
  year    = {2020},
  pages   = {842--850},
  doi     = {10.1007/s10114-020-9487-x}
}

@article{Belardo2020GraphsWA,
  author  = {Wang, Jian Feng and Wang, Jing and Liu, Xiaogang and Belardo, Francesco},
  title   = {Graphs whose {$A_\alpha$}-spectral radius does not exceed 2},
  journal = {Discussiones Mathematicae Graph Theory},
  volume  = {40},
  number  = {2},
  year    = {2020},
  pages   = {677--690},
  doi     = {10.7151/dmgt.2288}
}

@book{berge,
  author    = {Berge, Claude},
  title     = {Hypergraphs: Combinatorics of Finite Sets},
  series    = {North-Holland Mathematical Library},
  volume    = {45},
  publisher = {North-Holland},
  address   = {Amsterdam},
  year      = {1989}
}

@article{QI2013228,
  author  = {Liqun Qi},
  title   = {Symmetric nonnegative tensors and copositive tensors},
  journal = {Linear Algebra and its Applications},
  volume  = {439},
  number  = {1},
  year    = {2013},
  pages   = {228--238},
  issn    = {0024-3795},
  doi     = {10.1016/j.laa.2013.03.015}
}

@article{SHAO20132350,
  author  = {Jia-Yu Shao},
  title   = {A general product of tensors with applications},
  journal = {Linear Algebra and its Applications},
  volume  = {439},
  number  = {8},
  year    = {2013},
  pages   = {2350--2366},
  issn    = {0024-3795},
  doi     = {10.1016/j.laa.2013.07.010}
}

@article{MR4048055,
  author     = {Lin, Hongying and Guo, Haiyan and Zhou, Bo},
  title      = {On the {$\alpha$}-spectral radius of irregular uniform
              hypergraphs},
  journal    = {Linear Multilinear Algebra},
  fjournal   = {Linear and Multilinear Algebra},
  volume     = {68},
  number     = {2},
  year       = {2020},
  pages      = {265--277},
  issn       = {0308-1087,1563-5139},
  mrclass    = {05C50 (05C65)},
  mrnumber   = {4048055},
  mrreviewer = {Anna\ A.\ Taranenko},
  doi        = {10.1080/03081087.2018.1502253}
}

@incollection{MR347860,
  author     = {Hoffman, Alan J.},
  title      = {On limit points of spectral radii of non-negative symmetric
              integral matrices},
  booktitle  = {Graph theory and applications ({P}roc. {C}onf., {W}estern
              {M}ichigan {U}niv., {K}alamazoo, {M}ich., 1972; dedicated to
              the memory of {J}. {W}. {T}. {Y}oungs)},
  series     = {Lecture Notes in Math.},
  volume     = {303},
  publisher  = {Springer, Berlin-New York},
  year       = {1972},
  pages      = {165--172},
  mrclass    = {15A48 (05C99)},
  mrnumber   = {347860},
  mrreviewer = {H.\ Minc}
}

@article{GUO20172616,
  author   = {Krystal Guo and Bojan Mohar},
  title    = {Digraphs with {Hermitian} spectral radius below 2 and their cospectrality with paths},
  journal  = {Discrete Mathematics},
  volume   = {340},
  number   = {11},
  year     = {2017},
  pages    = {2616--2632},
  note     = {In Memory of Horst Sachs},
  issn     = {0012-365X},
  doi      = {10.1016/j.disc.2017.01.018}
}

@article{GavrilyukMunemasa2023,
  author  = {Gavrilyuk, Alexander L. and Munemasa, Akihiro},
  title   = {Maximal digraphs whose {Hermitian} spectral radius is at most 2},
  journal = {Linear Algebra and its Applications},
  volume  = {658},
  year    = {2023},
  pages   = {331--349},
  doi     = {10.1016/j.laa.2022.11.007}
}

@article{MR683990,
  author     = {Cvetkovi{\'c}, Drago{\v{s}} and Doob, Michael and Gutman, Ivan},
  title      = {On graphs whose spectral radius does not exceed
              {$(2+\sqrt{5})\sp{1/2}$}},
  journal    = {Ars Combin.},
  fjournal   = {Ars Combinatoria},
  volume     = {14},
  year       = {1982},
  pages      = {225--239},
  issn       = {0381-7032},
  mrclass    = {05C50},
  mrnumber   = {683990},
  mrreviewer = {Peter\ Kov\'acs}
}

@article{BrouwerNeumaier1989,
  author  = {Brouwer, A. E. and Neumaier, A.},
  title   = {The graphs with spectral radius between 2 and
             {$\sqrt{2+\sqrt{5}}$}},
  journal = {Linear Algebra and its Applications},
  volume  = {114--115},
  year    = {1989},
  pages   = {273--276},
  doi     = {10.1016/0024-3795(89)90466-7}
}

@article{MR4548954,
  author     = {Wang, Dijian and Dong, Wenkuan and Hou, Yaoping and Li,
              Deqiong},
  title      = {On signed graphs whose spectral radius does not exceed
              {$\sqrt{2 + \sqrt5}$}},
  journal    = {Discrete Math.},
  fjournal   = {Discrete Mathematics},
  volume     = {346},
  number     = {6},
  year       = {2023},
  pages      = {Paper No. 113358, 21},
  issn       = {0012-365X,1872-681X},
  mrclass    = {05C50},
  mrnumber   = {4548954},
  mrreviewer = {James\ McKee},
  doi        = {10.1016/j.disc.2023.113358}
}

@article{WANG20081606,
  author   = {Jianfeng Wang and Qiongxiang Huang and Xinhui An and Francesco Belardo},
  title    = {Some notes on graphs whose spectral radius is close to
         {$\frac{3}{2}\sqrt{2}$}},
  journal  = {Linear Algebra and its Applications},
  volume   = {429},
  number   = {7},
  year     = {2008},
  pages    = {1606--1618},
  issn     = {0024-3795},
  doi      = {10.1016/j.laa.2008.04.034}
}

@article{MR2365422,
  author     = {Woo, Renee and Neumaier, Arnold},
  title      = {On graphs whose spectral radius is bounded by
              {$\frac32\sqrt2$}},
  journal    = {Graphs Combin.},
  fjournal   = {Graphs and Combinatorics},
  volume     = {23},
  number     = {6},
  year       = {2007},
  pages      = {713--726},
  issn       = {0911-0119,1435-5914},
  mrclass    = {05C50},
  mrnumber   = {2365422},
  mrreviewer = {Mirko\ Lepovi\'c},
  doi        = {10.1007/s00373-007-0745-9}
}

@article{MR3034539,
  author   = {Lan, Jingfen and Lu, Linyuan},
  title    = {Diameters of graphs with spectral radius at most
              {$\frac{3}{2}\sqrt{2}$}},
  journal  = {Linear Algebra Appl.},
  fjournal = {Linear Algebra and its Applications},
  volume   = {438},
  number   = {11},
  year     = {2013},
  pages    = {4382--4407},
  issn     = {0024-3795,1873-1856},
  mrclass  = {05C50},
  mrnumber = {3034539},
  doi      = {10.1016/j.laa.2012.12.047}
}
\end{document}